\documentclass[11pt]{amsart}
\usepackage{geometry}
\usepackage{graphicx}
\usepackage{amssymb}
\usepackage{epstopdf}
\usepackage[colorlinks]{hyperref}
\usepackage[lite]{amsrefs}
\usepackage{mathpazo}
\usepackage{amsmath}
\usepackage{mathrsfs}
\usepackage{extarrows}
\newtheorem{theorem}{Theorem}[section]
\newtheorem{corollary}[theorem]{Corollary}
\newtheorem{lemma}[theorem]{Lemma}
\newtheorem{proposition}[theorem]{Proposition}
\theoremstyle{definition}
\newtheorem{definition}[theorem]{Definition}
\newtheorem{remark}[theorem]{Remark}
\newtheorem{example}[theorem]{Example}

\numberwithin{equation}{section}

\usepackage{tikz-cd}
\usepackage{tikz}
\usetikzlibrary{knots}
\usetikzlibrary{calc}
\usepgfmodule{decorations}
\usetikzlibrary{decorations.markings}
\usetikzlibrary{hobby}

\title{A categorification for the partial-dual genus polynomial}
\author{Zhiyun Cheng}
\address{School of Mathematical Sciences, Laboratory of Mathematics and Complex Systems, MOE, Beijing Normal University, Beijing 100875, China.}
\email{czy@bnu.edu.cn}
\author{Ziyi Lei}
\address{School of Mathematical Sciences, Beijing Normal University, Beijing 100875, China}
\email{202121130067@mail.bnu.edu.cn}
\subjclass[2020]{57M15, 05C10, 57K12}
\keywords{ribbon graph, partial-dual genus polynomial, TQFT, categorification}
\begin{document}
\begin{abstract}
The partial-dual genus polynomial $^\partial\varepsilon_G(z)$ of a ribbon graph $G$ is the generating function that enumerates all partial duals of $G$. In this paper, we give a categorification for this polynomial. The key ingredient of the construction is an extended Frobenius algebra related to unoriented topological quantum field theory.
\end{abstract}
\maketitle

\section{Introduction}\label{section1}
A ribbon graph is a surface with nonempty boundary, which consists of disjoint disks with some bands connecting them. Ribbon graphs occur in many areas of mathematics, such as topological graph theory, combinatorics, low-dimensional topology and representation theory. As a generalization of the classical Euler-Poincar\'{e} duality, Chmutov introduced the notion of partial duality in \cite{Chm2009}, which is a duality of ribbon graphs relative to a subset of edges. For a given ribbon graph $G$ and a subset $A\subseteq E(G)$, the partial dual of $G$ with respect to $A$ defines a new ribbon graph $G^A$. This duality satisfies several natural properties. For instance, we have $(G^A)^A=G$, $(G^A)^B=G^{(A\cup B)\setminus(A\cap B)}$ and the partial duality preserves orientability and the number of connected componets of ribbon graphs.

However, the genus of a ribbon graph may be changed under the partial duality. For this reason, Gross, Mansour and Tucker introduced the partial-dual genus polynomial in \cite{GMT2020}, which enumerates all possible partial duals of a ribbon graph by genus as represented by its generating function. This polynomial has been intensively studied during the past several years. For example, a concrete example was given in \cite{GMT2020} of which the partial-dual genus polynomial is not log-concave. Later in \cite{GMT2021} and \cite{GMT2021+}, Gross, Mansour and Tucker introduced some other related polynomials and presented a Gray code algorithm for calculating them. For bouquets, i.e. ribbon graphs with exactly one vertex-disk, it was proved in \cite{YJ2022} that the partial-dual genus polynomial only depends on the signed intersection graph. For ribbon graphs derived from chord diagrams, recently Chmutov proved that the partial-dual genus polynomial satisfies the four-term relation \cite{Chm2023}, which suggests a potential relationship between the partial-dual genus polynomial and finite type invariants. In particular, in this case the partial-dual genus polynomial also only depends on the intersection graph of the chord diagram, rather than the chord diagram itself \cite{YJ2022}.

Roughly speaking, categorification can be regarded as a lifting from an $n$-category to an $(n+1)$-category. One example is the lifting from the Euler characteristic of a smooth manifold to the homology groups of it. Since Khovanov's seminal work on the categorification of the Jones polynomial \cite{Kho2000}, many quantum invariants of knots have been categorified during the past twenty years, see \cite{Kho2004, Kho2005, KR2008, KR2008II} for some examples. These categorifications reinterpreted knot invariants as graded Euler characteristic of knot homologies. Several major breakthroughs based on the categorification have been made. For example, Rasmussen introduced his $s$-invariant and used it to give a purely combinatorial proof of Milnor conjecture and the existence of exotic $\mathbb{R}^4$ \cite{Ras2010}. Very recently, Piccirillo indirectly used Rasmussen's invariant to show that the Conway knot is not slice \cite{Pic2020}.

On the other hand, the similar idea was taken by Helme-Guizon and Rong in \cite{HR2005} to give a categorification of the chromatic polynomial of graphs. For each graph, they defined a bigraded cohomology theory such that the chromatic polynomial can be derived from the graded Euler characteristic of the cohomology groups. In particular, the deletion-contraction formula for the chromatic polynomial is lifted to a long exact sequence. Since then, many other graph polynomials have been categorified, see \cite{JR2006, VV2007, SY2018, CLWZ2022}. It turns out that usually these categorifications contain more information than the original polynomials. For example, the torsion part cannot be read from the polynomials.

 The main aim of this paper is to give a categorification of the partial-dual genus polynomial of ribbon graphs.

 This paper is organized as follows. In Section \ref{section2}, we take a quick review of the definitions and properties of ribbon graphs and partial duality. Then a formula for calculating the partial-dual genus polynomial of ribbon graphs is given. In particular, we introduce the notion of graded partial-dual genus polynomial, which can be regarded as a refined version of the partial-dual genus polynomial. Section \ref{section3} introduces a punctured $(1+1)$-TQFT, which plays an important role in the categorification of the partial-dual genus polynomial. The construction of the cochain complex involves four $n$-cubes, which will be discussed in Section \ref{section4}. As the main result of this paper, Theorem \ref{main theorem} is given in Section \ref{section5}, which shows that the graded partial-dual genus polynomial (hence also the partial-dual genus polynomial) can be obtained from the graded Euler characteristic of the cohomology groups. An concrete example is given in Section \ref{section6}, which suggests that the cohomology groups contains more information comparing with the partial-dual genus polynomial. We discuss some applications of this categorification in Section \ref{section7}, two ribbon graphs with the same partial-dual genus polynomial but distinct cohomology groups are provided. Finally, some algebraic structures related to this categorification are discussed in Section \ref{section8}.

\section{Ribbon graphs and partial-dual genus polynomial}\label{section2}
\subsection{Ribbon graphs and partial dual}\label{section2.1}
A \emph{ribbon graph} is a surface consists of finitely many disjoint vertex-disks and some edge-ribbons, such that the vertex-disks and edge-ribbons intersect by disjoint line segments and each line segment lies on the boundary of one vertex-disk and the boundary of one edge-ribbon. In particular, each edge-ribbon contains exactly two such line segments. Equivalently, given a cellular embedding of a graph in a 2-dimensional surface (not necessary orientable), then a regular neighborhood of this graph defines an associated ribbon graph. For a given ribbon graph $G$, let us use $V(G)$ and $E(G)$ to denote the set of vertex-disks of $G$ and the set of edge-ribbons of $G$, respectively. For simplicity, we sometimes just call the elements in $V(G)$ and $E(G)$ the \emph{vertices} of $G$ and the \emph{edges} of $G$. If we regard a ribbon graph $G$ as the regular neighborhood of a graph embedded in a surface, then we call each component of the complementary of $G$ a face-disk of $G$ and use $F(G)$ to denote the set of all face-disks. Some simple examples of ribbon graphs can be found in Figure \ref{figure}. Note that a ribbon graph is just an abstract surface with boundary, not an embedded surface in $\mathbb{R}^3$.

\begin{figure}[h]
\begin{tikzpicture}[scale=0.7, use Hobby shortcut, baseline=-3pt]
    \draw (55:1) arc (55:180-55:1);
    \draw (180-35:1) arc (180-35:360+35:1);
    \draw (55:1)..(90:2)..(180-55:1);
    \draw (35:1)..(90:2.3)..(180-35:1);
  \end{tikzpicture}
  \quad
  %第二个图
  \begin{tikzpicture}[scale=0.6, use Hobby shortcut, baseline=-3pt]
    \draw (-35:1) arc (-35:35:1);
    \draw (55:1) arc (55:360-55:1);
    \draw ($(180-35:1)+(3,0)$) arc (180-35:180+35:1);
    \draw ($(180-55:1)+(3,0)$) arc (180-55:-180+55:1);
    \draw (35:1)..(1.5,0.9)..($(180-35:1)+(3,0)$);
    \draw (55:1)..(1.5,1.2)..($(180-55:1)+(3,0)$);
    \begin{knot}[clip width=3.0,]
      \strand (-55:1)..(1.2,-1.2)..($(180+35:1)+(3,0)$);
      \strand (-35:1)..(1.8,-1.2)..($(180+55:1)+(3,0)$);
      \end{knot}
  \end{tikzpicture}
  \quad
  %第三个图
  \begin{tikzpicture}[scale=0.5, use Hobby shortcut, baseline=10pt]
    \draw (10:1) arc (10:50:1);
    \draw (70:1) arc (70:350:1);
    \draw ($(-170:1)+(4,0)$) arc (-170:110:1);
    \draw ($(130:1)+(4,0)$) arc (130:170:1);
    \draw ($(-70:1)+(60:4)$) arc (-70:-110:1);
    \draw ($(-50:1)+(60:4)$) arc (-50:230:1);
    \draw (10:1)..($(170:1)+(4,0)$);
    \draw (-10:1)..($(-170:1)+(4,0)$);
    \begin{knot}[clip width=3.0,]
      \strand (70:1)..++(60:0.01)..($(-110:1)+(60:4)$)..++(60:0.01);
      \strand (50:1)..++(60:0.01)..($(230:1)+(60:4)$)..++(60:0.01);
      \strand ($(-50:1)+(60:4)$)..++(-60:0.01)..($(130:1)+(4,0)$)..++(-60:0.01);
      \strand ($(-70:1)+(60:4)$)..++(-60:0.01)..($(110:1)+(4,0)$)..++(-60:0.01);
    \end{knot}
  \end{tikzpicture}
  \caption{Three examples of ribbon graphs}
  \label{figure}
\end{figure}
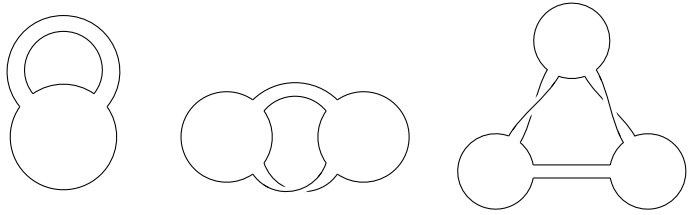

Before defining the partial dual of a ribbon graph, we first recall the classical Euler-Poincar\'{e} dual. Let $G$ be a ribbon graph, the \emph{Euler-Poincar\'{e} dual} $G^*$ can be obtained from $G$ by gluing a disk to each boundary component of $G$ and removing the interior of all vertex-disks. Now the newly glued disks form the set of vertex-disks of $G^*$ and the edge-ribbons of $G^*$ are exactly the same as that of $G$, but the attachments are changed. Notice that the genus of $G^*$ is equal to the genus of $G$, since they are the dual of each other in the same surface.

Now we turn to the partial dual of $G$. Fix a subset $A\subseteq E(G)$, the \emph{partial dual} $G^A$ with respect to the subset $A$ is defined as follows. Consider the spanning subgraph consisting of $V(G)$ and all edges in $A$, glue a disk to each boundary component of this spanning subgraph and remove the interior of all vertex-disks, then we obtain the partial dual $G^A$. Now the newly glued disks are the vertex-disks of $G^A$ and the edge-ribbons of $G^A$ are the same as $G$, but the attachments of the edges from $A$ are different from that in $G$. It is easy to observe that if $A=E(G)$, then $G^A=G^*$. The partial dual of a ribbon graph with respect to a subset of $E(G)$ also can be obtained from the so-called \emph{arrow presentation}. The reader is referred to \cite{Chm2009, Mof2013} for more details on the construction and some basic properties of the partial dual.

In particular, when the subset $A$ contains exactly one edge, say $A=\{e\}\subseteq E(G)$, the partial dual $G^A$ is shown in Figure \ref{figure0}. We remark that the direction of the arrow in Figure \ref{figure0} can be reversed, since $(G^{\{e\}})^{\{e\}}=G$. In general, if $|A|\geq2$, then $G^A$ can be obtained from $G$ by taking the partial dual on each edge in $A$ in sequence.

\begin{figure}
  $$\begin{tikzpicture}[scale=0.6, use Hobby shortcut, baseline=-3pt]
    \draw[densely dashed] (0,0) node {$A$} circle (1);
    \draw[densely dashed] (6,0) node {$B$} circle (1);
    \draw (2,0.5)--(4,0.5) (2,-0.5)--(4,-0.5);
    \node at (3,0) {$e$};
    \draw (2,1)--(2,-1) (4,1)--(4,-1);
    \draw (2,1) arc (0:180:1);
    \draw (2,-1) arc (0:-180:1);
    \draw (6,1) arc (0:180:1);
    \draw (6,-1) arc (0:-180:1);
    \node at (3,-2.7) {$G$};
  \end{tikzpicture}
  \quad\longrightarrow\quad
  \begin{tikzpicture}[scale=0.6, use Hobby shortcut, baseline=-3pt]
    \draw[densely dashed] (0,0) node {$A$} circle (1);
    \draw[densely dashed] (6,0) node {$B$} circle (1);
    \draw (3.5,2)--(3.5,-2) (2.5,2)--(2.5,-2);
    \node at (3,0) {$e$};
    \draw (1,2)--(5,2) (1,-2)--(5,-2);
    \draw (1,2) arc (90:180:1);
    \draw (1,-2) arc (-90:-180:1);
    \draw (6,1) arc (0:90:1);
    \draw (6,-1) arc (0:-90:1);
    \node at (3,-2.7) {$G^{\{e\}}$};
  \end{tikzpicture}$$
  $$\begin{tikzpicture}[scale=0.6, use Hobby shortcut, baseline=-3pt]
    \draw[densely dashed] (0,0) node {$A$} circle (1);
    \draw[densely dashed] (6,0) node {$B$} circle (1);
    \begin{knot}[clip width=3.0,]
      \strand (3.5,-2)..++(0,0.01)..(2.5,2)..++(0,0.01);
      \strand (2.5,-2)..++(0,0.01)..(3.5,2)..++(0,0.01);
    \end{knot}
    \node at (3,-1.5) {$e$};
    \draw (1,2)--(5,2) (1,-2)--(5,-2);
    \draw (1,2) arc (90:180:1);
    \draw (1,-2) arc (-90:-180:1);
    \draw (6,1) arc (0:90:1);
    \draw (6,-1) arc (0:-90:1);
    \node at (3,-2.7) {$G$};
  \end{tikzpicture}
  \quad\longrightarrow\quad
  \begin{tikzpicture}[scale=0.6, use Hobby shortcut, baseline=-3pt]
    \draw[densely dashed] (0,0) node {\rotatebox{180}{$A$}} circle (1);
    \draw[densely dashed] (6,0) node {$B$} circle (1);
    \begin{knot}[clip width=3.0,]
      \strand (3.5,-2)..++(0,0.01)..(2.5,2)..++(0,0.01);
      \strand (2.5,-2)..++(0,0.01)..(3.5,2)..++(0,0.01);
    \end{knot}
    \node at (3,-1.5) {$e$};
    \draw (1,2)--(5,2) (1,-2)--(5,-2);
    \draw (1,2) arc (90:180:1);
    \draw (1,-2) arc (-90:-180:1);
    \draw (6,1) arc (0:90:1);
    \draw (6,-1) arc (0:-90:1);
    \node at (3,-2.7) {$G^{\{e\}}$};
  \end{tikzpicture}$$
  \caption{Partial dual with respect to one edge-ribbon}
  \label{figure0}
\end{figure}
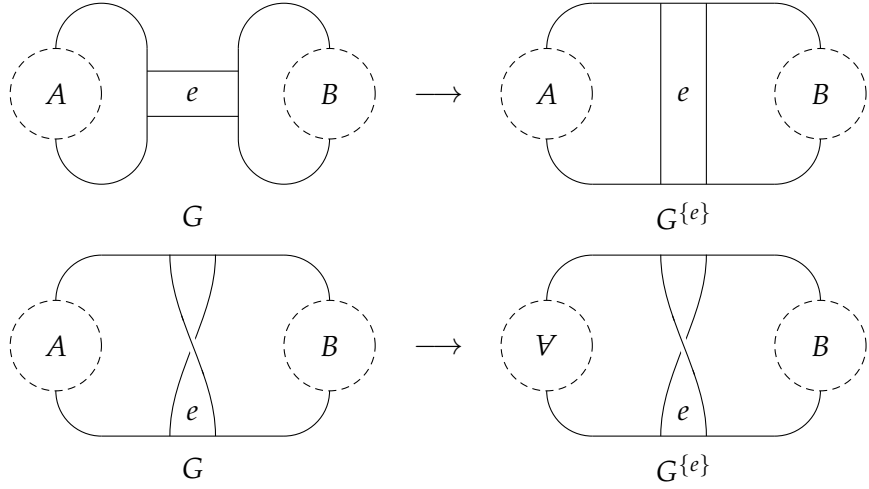

\subsection{Partial-dual genus polynomial}\label{section2.2}
Let $G$ be a ribbon graph, after gluing a disk to each boundary component of $G$, we obtain a closed surface $S$. The \emph{Euler genus} of $G$, denoted by $\varepsilon(G)$, is defined to be twice the genus of $S$ if $S$ is orientable, otherwise it is equal to the crosscap number of $S$. The Euler genus of $G$ can be calculated from the following formula
\begin{center}
$\varepsilon(G)=2c(G)-|V(G)|+|E(G)|-|F(G)|$,
\end{center}
where $c(G)$ denotes the number of components of $G$.

\begin{definition}[\cite{GMT2020}]\label{definition2.1}
For a ribbon graph $G$, denote the set of vertex-disks, edge-ribbons, face-disks of $G$ by $V(G), E(G)$ and $F(G)$ respectively, denote $c(G)$ to be the number of components of $G$. The \emph{partial-dual genus polynomial} of $G$ is defined as
\begin{center}
$^\partial\varepsilon_G(z)=\sum\limits_{A\subseteq E(G)}z^{\varepsilon(G^A)}$.
\end{center}
\end{definition}

Consider $A\subseteq E(G)$ as the ribbon subgraph obtained from $G$ by removing all edge-ribbons out of $A$ from $G$. Then we have the following combinatorial formula for calculating the partial-dual genus polynomial.

\begin{lemma}\label{combinatorial def}
For an arbitrary ribbon graph $G$, we have
\begin{center}
$^\partial\varepsilon_G(z)=z^{2c(G)+|E(G)|}\sum\limits_{A\subseteq E(G)}z^{-|F(A)|-|F(A^c)|}$,
\end{center}
where $A^c$ denotes $E(G)\setminus A$.
\end{lemma}
\begin{proof}
Fix a subset $A\subseteq E(G)$, the Euler genus of the partial dual $G^A$ equals
\begin{center}
$\varepsilon(G^A)=2c(G^A)-|V(G^A)|+|E(G^A)|-|F(G^A)|=2c(G)+|E(G)|-|F(A)|-|F(A^c)|$.
\end{center}
The second equality comes from the fact \cite[Theorem 2.1]{GMT2020} that $c(G)=c(G^A), |E(G)|=|E(G^A)|, |V(G^A)|=|F(A)|$ and $|F(G^A)|=|V((G^{A})^*)|=|V(G^{A^c})|=|F(A^c)|$. The result follows immediately.
\end{proof}

\begin{example}
We take the ribbon graph indexed by $11$ in Figure \ref{figure1} for an example. It has only one vertex-disk, which is the disk twisted to form a figure eight in the center, two edge-ribbons, which we order as the left one is the former one and right the latter one. It has four subgraphs, denoted by $00,\,01,\,10,$ and $11$ for short, where the index $1$ means to preserve corresponding edge and $0$ means to remove it. Note that $|F(00)|=|F(01)|=|F(11)|=1$ and $|F(10)|=2$, thus we have
$$^\partial\varepsilon_{11}(z)=z^{2\times 1+2}\cdot2(z^{-1-1}+z^{-1-2})=2z^2+2z.$$
\label{example 2.3}
\end{example}

\begin{figure}
  \centering
  \begin{tikzpicture}[scale=1,use Hobby shortcut]
\begin{knot}[clip width=3.7,]
\strand (0.5,0.5)..++(0,-0.01)..(0,0)..(-0.5,-0.5)..++(0,-0.01);
\strand (-0.5,0.5)..++(0,-0.01)..(0,0)..(0.5,-0.5)..++(0,-0.01);
\end{knot}
\draw (0.5,0.7) arc (0:180:0.5);
\draw (-0.5,-0.7) arc (180:360:0.5);
\draw (0.5,0.5)..++(0,0.2);
\draw (-0.5,0.5)..++(0,0.2);
\draw (0.5,-0.5)..++(0,-0.2);
\draw (-0.5,-0.5)..++(0,-0.2);
\node at (1,-1) {$00$};

\begin{knot}[clip width=3.7,]
  \strand (4.5+4,0.5)..++(0,-0.01)..(4+4,0)..(3.5+4,-0.5)..++(0,-0.01);
  \strand (3.5+4,0.5)..++(0,-0.01)..(4+4,0)..(4.5+4,-0.5)..++(0,-0.01);
  \strand (7.5,0.5) arc (90:270:0.6);
  \strand (7.5,0.7) arc (90:270:0.6);
  \end{knot}
  \draw (4.5+4,0.7) arc (0:180:0.5);
  \draw (3.5+4,-0.7) arc (180:360:0.5);
\draw (8.5,-0.5) arc (-90:90:0.5);
\draw (8.5,-0.7) arc (-90:90:0.7);
\node at (9,-1) {$11$};

\begin{knot}[clip width=3.7,]
    \strand (2.5+2,2)..++(0,-0.01)..(2+2,1.5)..(+2+1.5,1)..++(0,-0.01);
    \strand (2-0.5+2,2)..++(0,-0.01)..(2+2,1.5)..(2.5+2,1)..++(0,-0.01);
    \end{knot}
    \draw (2.5+2,2.2) arc (0:180:0.5);
    \draw (2-0.5+2,1-0.2) arc (180:360:0.5);
    \draw (8.5-4,-0.5+1.5) arc (-90:90:0.5);
    \draw (8.5-4,-0.7+1.5) arc (-90:90:0.7);
    \draw (3.5,2)..++(0,0.2);
    \draw (3.5,1)..++(0,-0.2);
    \node at (5,0.5) {$01$};

\begin{knot}[clip width=3.7,]
      \strand (2.5+2,-1)..++(0,-0.01)..(2+2,-1.5)..(2-0.5+2,-1.5-0.5)..++(0,-0.01);
      \strand (2-0.5+2,-1)..++(0,-0.01)..(2+2,-1.5)..(2.5+2,-2)..++(0,-0.01);
      \strand (3.5,-1) arc (90:270:0.6);
      \strand (3.5,-0.8) arc (90:270:0.6);
      \end{knot}
      \draw (2.5+2,0.7-1.5) arc (0:180:0.5);
      \draw (1.5+2,-0.7-1.5) arc (180:360:0.5);
      \draw (4.5,-2)..++(0,-0.2);
      \draw (4.5,-1)..++(0,0.2);
      \node at (5,-2.5) {$10$};

\node at (9,1) {$2$};
\node at (7,1) {$1$};
\end{tikzpicture}
\caption{A ribbon graph (indexed by $11$) and its three proper subgraphs}
\label{figure1}
\end{figure}
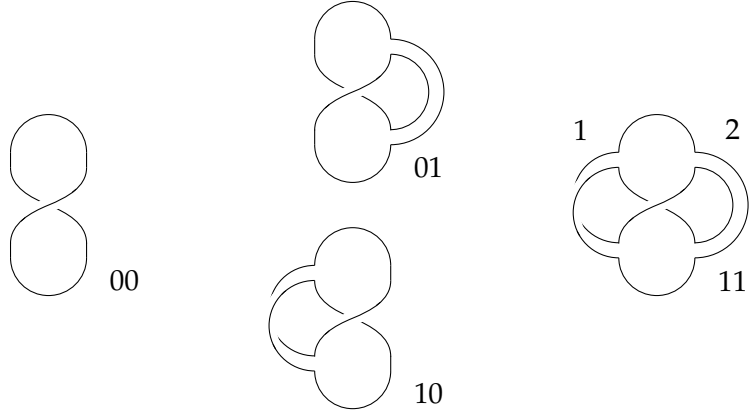

\subsection{Graded partial-dual genus polynomial}\label{section2.3}
For a given ribbon graph $G$, the component number $c(G)$ and edge number $|E(G)|$ are fixed, thus the essential part of the partial-dual genus polynomial $^\partial\varepsilon_G(z)$ is the sum $\sum\limits_{A\subseteq E(G)}z^{-|F(A)|-|F(A^c)|}$. Let us use $e_G(z)$ to denote $\sum\limits_{A\subseteq E(G)}z^{-|F(A)|-|F(A^c)|}$. Therefore, the partial-dual genus polynomial can be written as
\begin{center}
$^\partial\varepsilon_G(z)=z^{2c(G)+|E(G)|}e_G(z)$.
\end{center}
Now we introduce a polynomial of ribbon graphs in two variables, which can be considered as a graded version of the partial-dual genus polynomial.

\begin{definition}\label{definition2.4}
The \emph{graded partial-dual genus polynomial} of any ribbon graph $G$ is defined as
\begin{center}
$^\partial\tilde{\varepsilon}_G(w, z)=\sum\limits_{A\subseteq E(G)}w^{|A|}z^{\varepsilon(G^A)}=z^{2c(G)+|E(G)|}\sum\limits_{A\subseteq E(G)}w^{|A|}z^{-|F(A)|-|F(A^c)|}$.
\end{center}
\end{definition}

The graded partial-dual genus polynomial $^\partial\tilde{\varepsilon}_G(w, z)$ is a refined version of the partial-dual genus polynomial $^\partial\varepsilon_G(z)$, which splits $^\partial\varepsilon_G(z)$ into $|E(G)|+1$ parts with respect to the number of edges in the subset $A$. Just like the partial-dual genus polynomial, this graded version is completely determined by the sum $\sum\limits_{A\subseteq E(G)}w^{|A|}z^{-|F(A)|-|F(A^c)|}$, up to a $z$-degree shift. Let us set $\tilde{e}_G(w, z)=\sum\limits_{A\subseteq E(G)}w^{|A|}z^{-|F(A)|-|F(A^c)|}$, then we have
\begin{center}
$^\partial\tilde{\varepsilon}_G(w, z)=z^{2c(G)+|E(G)|}\tilde{e}_G(w, z), \tilde{e}_G(1, z)=e_G(z)$, and $^\partial\tilde{\varepsilon}_G(1, z)=^\partial\varepsilon_G(z)$.
\end{center}

\begin{example}\label{example2.5}
Consider the ribbon graph $G$ indexed by 11 in Figure \ref{figure1}, direct calculation shows that
\begin{center}
$\tilde{e}_G(w, z)=z^{-2}+2wz^{-3}+w^2z^{-2}$ and $^\partial\tilde{\varepsilon}_G(w, z)=z^{2\times1+2}\tilde{e}_G(w, z)=z^2+2zw+z^2w^2$.
\end{center}
\end{example}

In general, the coefficients of a partial-dual genus polynomial $^\partial\varepsilon_G(z)$ is not symmetric, see \cite[Example 3.1]{GMT2020} for an example. However, if the graded partial-dual genus polynomial of a ribbon graph $G$ has the form
\begin{center}
$^\partial\tilde{\varepsilon}_G(w, z)=\sum\limits_{i=0}^{|E(G)|}f_i(z)w^i$,
\end{center}
it is easy to conclude from the definition that $f_i(z)=f_{n-i}(z)$.

\section{A punctured (1+1)D-TQFT}\label{section3}
The aim of this section is to introduce a punctured (1+1)D-TQFT, which will be used in the construction of the categorification of the graded partial-dual genus polynomial.

Let $M=\mathbb{Z}[\sqrt{3},x]/<x^3>$ over the ring $\mathbb{Z}[\sqrt{3}]$, equipped with the unit
\begin{center}
$u: \mathbb{Z}[\sqrt{3}]\to M,\,1\mapsto 1$,
\end{center}
nature multiplication, Frobenius trace
\begin{center}
$\epsilon: M\to\mathbb{Z}[\sqrt{3}],\,1,x\mapsto 0;\,x^2\mapsto 1$,
\end{center}
and a half genus map
\begin{center}
$h: M\to M,\,1\mapsto\sqrt{3}x;\,x\mapsto \sqrt{3}x^2;\,x^2\mapsto 0$,
\end{center}
which can be seen as a multiplication with $\sqrt{3}x$. It is a commutative Frobenius algebra and the Frobenius trace induces a unique comultiplication
\begin{align*}
  \Delta: M&\mapsto M\otimes M,\\
         1&\mapsto 1\otimes x^2+x\otimes x+x^2\otimes 1,\\
         x&\mapsto x\otimes x^2+x^2\otimes x,\\
         x^2&\mapsto x^2\otimes x^2.
\end{align*}
Note that throughout this paper, the notation $\otimes$ always denotes the tensor product over the ring $\mathbb{Z}[\sqrt{3}]$.

Now we make $M$ into a graded ring by assigning
\begin{center}
deg $1=1$, deg $x=0$ and deg $x^2=-1$.
\end{center}
Then the multiplication $m$, half genus map $h$ and comultiplication $\Delta$ are all of degree $-1$, while the unit $u$ and trace $\epsilon$ are of degree $1$.

\begin{remark}
Here we use a different degree and an opposite trace comparing with the the Frobenius algebra used by Khovanov in the definition of $sl(3)$ link homology \cite{Kho2004}. The reason for the former is that we do not need to realize $M$ as the cohomology ring of some concrete space, neither hope that our half genus map will introduce $\sqrt{-1}$.
\end{remark}

The commutative Frobenius algebra $M$ equipped with the half genus map $h$ gives rise to a functor $\mathcal{F}$ from the category of $2$-dimensional oriented cobordisms equipped with punctured points to the category of graded abelian groups. On objects, $\mathcal{F}$ is given by $\mathcal{F}(M^1)=M^{\otimes J}$ where $J$ is the set of components of a 1-dimensional manifold $M^1$. On morphisms, for each basic cobordism occurring in the first row of Figure~\ref{figure2}, $\mathcal{F}$ maps it into the map listed below.

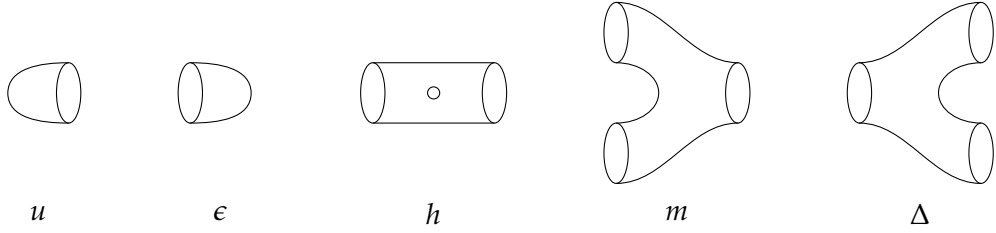
\begin{figure}[htbp]
  \centering
  \begin{tikzpicture}[scale=0.8, use Hobby shortcut]
    %unit
    \draw (1,0.5)..++(-0.01,0)..(0,0)..++(0,-0.01)..(1,-0.5)..++(0.01,0);
    \draw (1,0) ellipse (0.2 and 0.5);
    \node at (0.5,-2) {$u$};
    %trace
    \draw (3,0.5)..++(0.01,0)..(4,0)..++(0,-0.01)..(3,-0.5)..++(-0.01,0);
    \draw (3,0) ellipse (0.2 and 0.5);
    \node at (3.5,-2) {$\epsilon$};
    %half genus map
    \draw (6,0.5)..(8,0.5) (6,-0.5)..(8,-0.5);
    \draw (8,0) ellipse (0.2 and 0.5);
    \draw (6,0) ellipse (0.2 and 0.5);
    \draw (7,0) circle (.1);
    \node at (7,-2) {$h$};
    %multiplication
    \draw (10,1.5)..++(0.01,0)..(12,0.5)..++(0.01,0);
    \draw (10,-1.5)..++(0.01,0)..(12,-0.5)..++(0.01,0);
    \draw (10,-0.5) arc (-90:90:0.7 and 0.5);
    \draw (10,1) ellipse (0.2 and 0.5);
    \draw (10,-1) ellipse (0.2 and 0.5);
    \draw (12,0) ellipse (0.2 and 0.5);
    \node at (11,-2) {$m$};
    %comultiplication
    \draw (16,1.5)..++(-0.01,0)..(14,0.5)..++(-0.01,0);
    \draw (16,-1.5)..++(-0.01,0)..(14,-0.5)..++(-0.01,0);
    \draw (16,0.5) arc (90:270:0.7 and 0.5);
    \draw (16,1) ellipse (0.2 and 0.5);
    \draw (16,-1) ellipse (0.2 and 0.5);
    \draw (14,0) ellipse (0.2 and 0.5);
    \node at (15,-2) {$\Delta$};
  \end{tikzpicture}
  \caption{Basic cobordisms and corresponding maps}
  \label{figure2}
\end{figure}

The name for the half genus map comes from the following relation
\begin{equation}\label{h2=mD}
  \mathcal{F}(\ \begin{tikzpicture}[scale=0.8, use Hobby shortcut, baseline=-3pt]
      \draw (6,0.5)..(9,0.5) (6,-0.5)..(9,-0.5);
      \draw (9,0) ellipse (0.2 and 0.5);
      \draw (6,0) ellipse (0.2 and 0.5);
      \draw (7,0) circle (.1);
      \draw (8,0) circle (.1);
    \end{tikzpicture} \ )
  =\mathcal{F}(\  \begin{tikzpicture}[scale=0.8, use Hobby shortcut, baseline=-3pt]
    \draw (10,1.5)..++(0.01,0)..(12,0.5)..++(0.01,0);
    \draw (10,-1.5)..++(0.01,0)..(12,-0.5)..++(0.01,0);
    \draw (12,0) ellipse (0.2 and 0.5);
    \draw (10,1.5)..++(-0.01,0)..(8,0.5)..++(-0.01,0);
    \draw (10,-1.5)..++(-0.01,0)..(8,-0.5)..++(-0.01,0);
    \draw (8,0) ellipse (0.2 and 0.5);
    \draw (10,0) ellipse (0.7 and 0.5);
  \end{tikzpicture} \ ),
\end{equation}
or algebraically,
$$h^2=m\circ\Delta.$$
That is to say, the addition of two punctures is equal to the addition of one genus.

\begin{remark}
Khovanov's $sl(3)$ link homology uses a similar morphism which is decorated by a dot, and $\mathcal{F}(
\begin{tikzpicture}[scale=0.6, use Hobby shortcut, baseline=-3pt]
    \draw (6,0.5)..(8,0.5) (6,-0.5)..(8,-0.5);
    \draw (8,0) ellipse (0.2 and 0.5);
    \draw (6,0) ellipse (0.2 and 0.5);
    \draw[fill] (7,0) circle (.1);
\end{tikzpicture} )$ denotes the multiplication with $x$, being different from our half genus map by a coefficient $\sqrt{3}$. To emphasis this difference but inherit the similarity, here we use the notation of punctures.
\end{remark}

Naturally, the punctures on cobordisms should be able to move freely, with the morphisms these cobordisms represent preserved. It follows from the following relations.
\begin{equation}\label{puncture move 1}
  \mathcal{F}( \
\begin{tikzpicture}[scale=0.6, use Hobby shortcut, baseline=-3pt]
  \draw (10,1.5)..++(0.01,0)..(12,0.5)..++(0.01,0);
    \draw (10,-1.5)..++(0.01,0)..(12,-0.5)..++(0.01,0);
    \draw (10,-0.5) arc (-90:90:0.7 and 0.5);
    \draw (10,1) ellipse (0.2 and 0.5);
    \draw (10,-1) ellipse (0.2 and 0.5);
    \draw (8,1) ellipse (0.2 and 0.5);
    \draw (8,-1) ellipse (0.2 and 0.5);
    \draw (12,0) ellipse (0.2 and 0.5);
    \draw (8,1.5)..(10,1.5);
    \draw (8,0.5)..(10,0.5);
    \draw (8,-1.5)..(10,-1.5);
    \draw (8,-0.5)..(10,-0.5);
    \draw (9,1) circle (.1);
\end{tikzpicture} \ )=\mathcal{F}( \
\begin{tikzpicture}[scale=0.6, use Hobby shortcut, baseline=-3pt]
  \draw (10,1.5)..++(0.01,0)..(12,0.5)..++(0.01,0);
    \draw (10,-1.5)..++(0.01,0)..(12,-0.5)..++(0.01,0);
    \draw (10,-0.5) arc (-90:90:0.7 and 0.5);
    \draw (10,1) ellipse (0.2 and 0.5);
    \draw (10,-1) ellipse (0.2 and 0.5);
    \draw (12,0) ellipse (0.2 and 0.5);
    \draw (14,0) ellipse (0.2 and 0.5);
    \draw (12,0.5)..(14,0.5);
    \draw (12,-0.5)..(14,-0.5);
    \draw (13,0) circle (.1);
\end{tikzpicture} \ )=\mathcal{F}( \ \begin{tikzpicture}[scale=0.6, use Hobby shortcut, baseline=-3pt]
  \draw (10,1.5)..++(0.01,0)..(12,0.5)..++(0.01,0);
    \draw (10,-1.5)..++(0.01,0)..(12,-0.5)..++(0.01,0);
    \draw (10,-0.5) arc (-90:90:0.7 and 0.5);
    \draw (10,1) ellipse (0.2 and 0.5);
    \draw (10,-1) ellipse (0.2 and 0.5);
    \draw (8,1) ellipse (0.2 and 0.5);
    \draw (8,-1) ellipse (0.2 and 0.5);
    \draw (12,0) ellipse (0.2 and 0.5);
    \draw (8,1.5)..(10,1.5);
    \draw (8,0.5)..(10,0.5);
    \draw (8,-1.5)..(10,-1.5);
    \draw (8,-0.5)..(10,-0.5);
    \draw (9,-1) circle (.1);
\end{tikzpicture} \ ),
\end{equation}
\begin{equation}\label{puncture move 2}
  \mathcal{F}( \
\begin{tikzpicture}[scale=0.6, use Hobby shortcut, baseline=-3pt]
  \draw (-10,1.5)..++(-0.01,0)..(-12,0.5)..++(-0.01,0);
    \draw (-10,-1.5)..++(-0.01,0)..(-12,-0.5)..++(-0.01,0);
    \draw (-10,0.5) arc (90:270:0.7 and 0.5);
    \draw (-10,1) ellipse (0.2 and 0.5);
    \draw (-10,-1) ellipse (0.2 and 0.5);
    \draw (-8,1) ellipse (0.2 and 0.5);
    \draw (-8,-1) ellipse (0.2 and 0.5);
    \draw (-12,0) ellipse (0.2 and 0.5);
    \draw (-8,1.5)..(-10,1.5);
    \draw (-8,0.5)..(-10,0.5);
    \draw (-8,-1.5)..(-10,-1.5);
    \draw (-8,-0.5)..(-10,-0.5);
    \draw (-9,1) circle (.1);
\end{tikzpicture} \ )=\mathcal{F}( \
\begin{tikzpicture}[scale=0.6, use Hobby shortcut, baseline=-3pt]
  \draw (-10,1.5)..++(-0.01,0)..(-12,0.5)..++(-0.01,0);
    \draw (-10,-1.5)..++(-0.01,0)..(-12,-0.5)..++(-0.01,0);
    \draw (-10,0.5) arc (90:270:0.7 and 0.5);
    \draw (-10,1) ellipse (0.2 and 0.5);
    \draw (-10,-1) ellipse (0.2 and 0.5);
    \draw (-12,0) ellipse (0.2 and 0.5);
    \draw (-14,0) ellipse (0.2 and 0.5);
    \draw (-12,0.5)..(-14,0.5);
    \draw (-12,-0.5)..(-14,-0.5);
    \draw (-13,0) circle (.1);
\end{tikzpicture} \ )=\mathcal{F}( \
\begin{tikzpicture}[scale=0.6, use Hobby shortcut, baseline=-3pt]
  \draw (-10,1.5)..++(-0.01,0)..(-12,0.5)..++(-0.01,0);
  \draw (-10,-1.5)..++(-0.01,0)..(-12,-0.5)..++(-0.01,0);
  \draw (-10,0.5) arc (90:270:0.7 and 0.5);
  \draw (-10,1) ellipse (0.2 and 0.5);
  \draw (-10,-1) ellipse (0.2 and 0.5);
  \draw (-8,1) ellipse (0.2 and 0.5);
  \draw (-8,-1) ellipse (0.2 and 0.5);
  \draw (-12,0) ellipse (0.2 and 0.5);
  \draw (-8,1.5)..(-10,1.5);
  \draw (-8,0.5)..(-10,0.5);
  \draw (-8,-1.5)..(-10,-1.5);
  \draw (-8,-0.5)..(-10,-0.5);
  \draw (-9,-1) circle (.1);
\end{tikzpicture} \ ),
\end{equation}
or algebraically,
$$m\circ(h\otimes\mathrm{Id}_M)=h\circ m=m\circ(\mathrm{Id}_M\otimes h),$$
$$(h\otimes\mathrm{Id}_M)\circ\Delta=\Delta\circ h=(\mathrm{Id}_M\otimes h)\circ\Delta.$$
Concisely speaking, $h$ is a self-dual $M$-module endomorphism, thus it is a multiplication by an element. These relations play an important role in our construction of commutative cubes.

Some other relations, especially the Frobenius relation
\begin{center}
$(\mathrm{Id}_M\otimes m)\circ(\Delta\circ\mathrm{Id}_M)=\Delta\circ m=(m\otimes\mathrm{Id}_M)\circ(\mathrm{Id}_M\otimes\Delta)$,
\end{center}
are needed to prove that our cubes introduced later is commutative, while these relations hold for any Frobenius algebra. More details can be found in \cite{Kock2004}.

\begin{remark}
The reader who is familiar with topological quantum field theory, especially unoriented topological quantum field theory, may have realized that the algebraic structure we introduced in this section is not new. In order to construct a link homology theory for stable equivalence classes of links in thickened orientable surfaces, the notion of unoriented topological quantum field theory was introduced by Turaev and Turner in \cite{TT2006}. Another closely related idea, the Klein topological field theory, was proposed by Alexeevski and Natanzon in the context of open-closed field theory \cite{AN2006}. The isomorphisms classes of (1+1)D unoriented topological quantum theories was classified in terms of extended Frobenius algebras. The algebraic structure we discussed here happens to be a special case of \emph{extended Frobenius algebra} in \cite{TT2006} and \emph{structure algebra} in \cite{AN2006}. More precisely, by setting $N=3$ in \cite[Example 2.6]{TT2006}, we obtain the Frobenius algebra $M$ above. We will continue the discussion of other possible choices of algebraic structures in Section \ref{section8}.
\end{remark}

\section{Commutative $n$-cube category and four cubes}\label{section4}
\subsection{Commutative n-cube category}
A $n$-cube, geometrically, is a $n$-dimensional unit cube located in a $n$-dimensional rectangular coordinate system. For each vertex $A\in\{0,1\}^n$, we assign a linear space $V_A$ and call it the \emph{vertex space} at $A$. As usual, we use $|A|$ to denote the number of 1's in $A$. For each edge $l$ connecting $A,B\in\{0,1\}^n$, where $B$ can obtained from $A$ by replacing a 0 with 1 (we simply say $A, B$ satisfy the \emph{edge condition}), we assign a linear map $\varphi_A^B: V_A\to V_B$ and call it the \emph{edge map} from $A$ to $B$. In this way we obtain an algebraic cube $(V,\varphi)$. If for an arbitrary $2$-dimensional face of $(V,\varphi)$, say
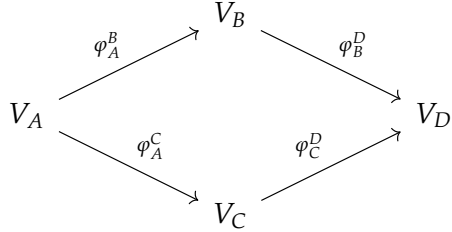
\begin{equation*}
  \begin{tikzcd}
    &  & V_B \arrow[rrd, "\varphi_B^D"] &  & \\
V_A \arrow[rru, "\varphi_A^B"] \arrow[rrd, "\varphi_A^C"] & & & & V_D \\
    &  & V_C \arrow[rru, "\varphi_C^D"] &  &
\end{tikzcd}
\end{equation*}
we have $\varphi_B^D\circ\varphi_A^B=\varphi_C^D\circ\varphi_A^C$, then we say $(V,\varphi)$ is \emph{commutative}. We say $(V,\varphi)$ is \emph{anti-commutative} if $\varphi_B^D\circ\varphi_A^B=-\varphi_C^D\circ\varphi_A^C$.

If there is a set of linear maps between the corresponding vertex spaces of two $n$-cube, being commutative with the edge maps, we call it a \emph{$n$-cube morphism}. The \emph{commutative $n$-cube category} is a category which consists of a collection of commutative $n$-cubes as objects and a collection of $n$-cube morphisms as morphisms. The reader is referred to \cite{Bar2002} for more details.

The commutative $n$-cube category can be equipped with a nature tensor operator $\otimes$, defined as below. For any two commutative $n$-cubes $(V,\varphi)$ and $(W,\phi)$, we define
\begin{align*}
   (V,\varphi)\otimes(W,\phi):=&(V\otimes W,\varphi\otimes\phi);\\
   (V\otimes W)_A:=&V_A\otimes W_A,\,\forall A\in\{0,1\}^n;\\
   (\varphi\otimes\phi)_A^B:=&\varphi_A^B\otimes\phi_A^B,\,\forall A, B\in\{0,1\}^n \text{ satisfy the edge condition.}
\end{align*}
It is straightforward to verify that $(V\otimes W,\varphi\otimes\phi)$ is still a commutative $n$-cube, which means that $\otimes$ is a well defined operator in the commutative $n$-cube category.

\begin{remark}\label{tensorProp}
The definition of $\otimes$ can be extended to any two algebraic cubes, including anti-commutative cubes. It is easy to check that for any two algebraic $n$-cubes $(V,\varphi)$ and $(W,\phi)$, if both $(V,\varphi)$ and $(W,\phi)$ are anti-commutative, then $(V,\varphi)\otimes(W,\phi)$ is commutative. And, if one of $(V,\varphi)$ and $(W,\phi)$ is commutative and the other one is anti-commutative, then $(V,\varphi)\otimes(W,\phi)$ is anti-commutative.
\end{remark}

In the remaining part of this section, we will construct one anti-commutative cube and three commutative cubes based on ribbon graphs. According to Remark \ref{tensorProp}, we can take the tensor product of them to obtain an anti-commutative cube. A cochain complex derived from this anti-commutative cube will be given in the next section.

\subsection{The anti-commutative $n$-cube $(S, s)$}
From now on, we always use $n$ to denote the number of the edge-ribbons in a given ribbon graph $G$. For arbitrary $A, B, C, D\in\{0,1\}^n$ such that $AB, AC, BD, CD$ are adjacent, let $m_A^B$ denote the number of $1$'s in the front of the element which is 0 in $A$ and 1 in $B$. The integers $m_A^C, m_B^D$ and $m_C^D$ can be defined similarly. It is easy to find that
$$m_A^B+m_B^D-m_A^C-m_C^D=\pm1,$$
which guarantees that the cube $(S, s)$ defined below is anti-commutative.

\begin{definition}
There is an anti-commutative $n$-cube $(S, s)$, where for each $A\in\{0,1\}^n$, the vertex space $S_A=\mathbb{Z}[\sqrt{3}]$, and for each $A, B\in\{0,1\}^n$ satisfy the edge condition, the edge map $s_A^B$ is the multiplication with $(-1)^{m_A^B}$. For brevity's sake, we call it the \emph{$S$-cube}.
\end{definition}

\begin{example}
Let $G$ be the ribbon graph indexed by $11$ in Example \ref{figure1}, the $S$-cube of G is shown in the Figure \ref{figure3}.
\end{example}

\begin{figure}[htbp]
\centering
\begin{tikzcd}
    &  & {\mathbb{Z}[\sqrt{3}]} \arrow[rrd, "1"] & & \\
{\mathbb{Z}[\sqrt{3}]} \arrow[rru, "1"] \arrow[rrd, "1"] & & & & {\mathbb{Z}[\sqrt{3}]} \\
    &  & {\mathbb{Z}[\sqrt{3}]} \arrow[rru, "-1"]  &  &
\end{tikzcd}
\caption{A $S$-cube}
\label{figure3}
\end{figure}
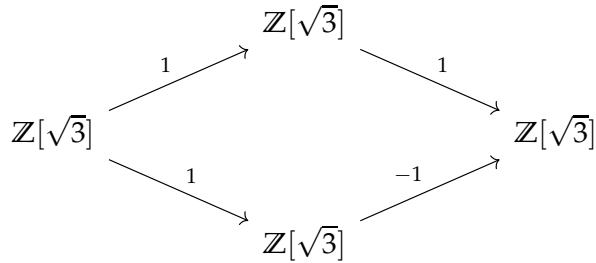

We choose the letter S because it is the first letter of sign. When one takes the tensor product of a commutative $n$-cube with the $S$-cube, it just adds some signs to the edge maps of the original cube, translating it into an anti-commutative $n$-cube.

\subsection{The commutative $n$-cube $(F, f)$}
In this subsection, we introduce another graded algebra $N=\mathbb{Z}[\sqrt{3}, y]/<y^2>$, where the degree of 1 and $y$ are equal to 0 and 1, respectively. This makes the edge map of the cube defined below has degree 0.
\begin{definition}
There is a commutative $n$-cube $(F, f)$, where for each $A\in\{0,1\}^n$ we set $F_A=N^{\otimes|A|}$, and for each $A, B\in\{0, 1\}^n$ which satisfy the edge condition, $f_A^B$ is defined as $u\otimes\mathrm{Id}_N^{\otimes |A|}$, say
$$f_A^B:N^{\otimes|A|}\to N^{\otimes(|A|+1)}, \alpha\mapsto 1\otimes\alpha.$$
Briefly, we call this cube the \emph{$F$-cube}.
\end{definition}

\begin{remark}
The main aim of introducing this cube is to erase the alternating terms of the graded (quantum) Euler characteristic of the cochain complex given later. This method was also used by Vershinin and Vesnin in the categorification of the Yamada polynomial \cite{VV2007}. Actually, $F_A$ can be chosen as any graded $\mathbb{Z}[\sqrt{3}]$-module and $f_A^B$ can be chosen as any other graded morphism. While the construction of this definition is select to make the cohomology groups we obtain not only easy to calculate but also preserve as much information as possible.
\end{remark}

\begin{example}
For the ribbon graph indexed by 11 in Figure \ref{figure1}, the $F$-cube is depicted in the Figure \ref{figure4}.
\end{example}

\begin{figure}[htbp]
\centering
\begin{tikzcd}
  &  & N \arrow[rrd, "u\otimes\mathrm{Id}_N"] &  & \\
{\mathbb{Z}[\sqrt{3}]} \arrow[rru, "u"] \arrow[rrd, "u"] &  &  &  & N^{\otimes 2}\\
  &  & N\arrow[rru, "u\otimes\mathrm{Id}_N"] &  &
\end{tikzcd}
\caption{A $F$-cube}
\label{figure4}
\end{figure}
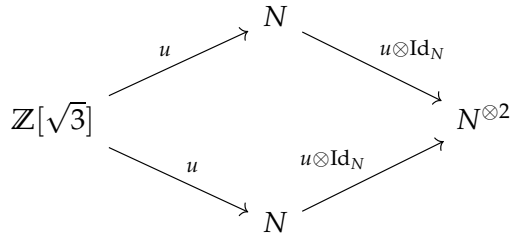

\subsection{The commutative $n$-cube $(V, v)$}
Fix an order of the edge-ribbons of the ribbon graph $G$, so that any $A\in\{0, 1\}^n$ can be seen as a subset of $E(G)$, where 0 means reject the corresponding edge-ribbon and 1 means include the corresponding one. As before, we also use $A$ to refer to the subgraph obtained from $G$ by removing the edges not in $A$.

Now we give the definition of the commutative cube $(V, v)$.
\begin{definition}\label{Def of $V$-cube}
For each $A\in\{0,1\}^n$, we define $V_A=M^{\otimes F(A)}$. For each $A, B\in\{0,1\}^n$ satisfying the edge condition, if
\begin{itemize}
\item $|F(B)|=|F(A)|-1$, then the adding of the new edge-ribbon merges two circles into one circle, and we assign a multiplication $m$ on $M\otimes M$ corresponding to these two circles and identity $\mathrm{Id}_M$ to other circles and take the tensor product of these operators together to obtain the edge map $v_A^B$.
\item $|F(B)|=|F(A)|+1$, then the adding of the new edge-ribbon splits one circle into two circles, and we assign a comultiplication $\Delta$ on $M$ corresponding to this circle and identity $\mathrm{Id}_M$ to other circles and take the tensor product of these operators together to obtain the edge map $v_A^B$.
\item $|F(B)|=|F(A)|$, then the adding of the new edge-ribbon translates one circle into another one, and we assign a half genus map $h$ on $M$ corresponding to this circle and identity $\mathrm{Id}_M$ to other circles and take the tensor product of these operators together to obtain the edge map $v_A^B$.
\end{itemize}
We named this cube the \emph{$V$-cube} for short.
\end{definition}

\begin{lemma}\label{lemma4.8}
The $V$-cube is commutative.
\end{lemma}
\begin{proof}
If the half genus map $h$ is not involved in a 2-dimensional face of $(V, v)$, the commutativity follows from the associativity of the multiplication, the coassociativity of the comultiplication and the Frobenius relation. If $h$ is involved, the commutativity can be derived from Equation \ref{h2=mD}, Equation \ref{puncture move 1} and Equation \ref{puncture move 2}.
\end{proof}

\begin{remark}
The first and second cases appear in Khovanov's approach to the categorification of the Jones polynomial \cite{Kho2000}, and the third case occurs because here these circles are no longer restricted on $\mathbb{S}^2$ or $\mathbb{R}^2$. In order to address this new problem, we have to propose the half genus map $h$.
\end{remark}

\begin{example}
For the ribbon graph indexed by $11$ in Figure \ref{figure1}, observing that $|F(00)|=|F(01)|=|F(11)|=1$ and $|F(10)|=2$, we know that the spaces assigned to $00,\,01,\,11$ are $M$ and the space associated to $10$ is $M\otimes M$, and the edge maps is shown in Figure \ref{figure5}. It is commutative as $m\circ \Delta=h^2$.
\end{example}

\begin{figure}[htbp]
\centering
\begin{tikzcd}
    &  & M \arrow[rrd, "h"] &  & \\
M \arrow[rrd, "\Delta"] \arrow[rru, "h"] & & & & M \\
    &  & M\otimes M \arrow[rru, "m"] &  &
\end{tikzcd}
\caption{A $V$-cube}
\label{figure5}
\end{figure}
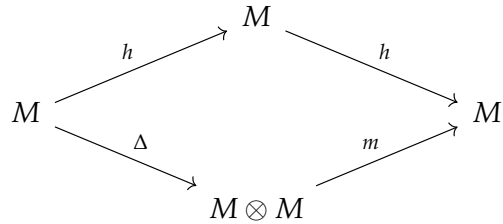

\subsection{The commutative $n$-cube $(W, w)$}
\begin{definition}
For each $A\in\{0, 1\}^n$, we set $W_A=M^{\otimes F(A^c)}$, where $A^c=E(G)\setminus A$. And for each $A, B\in\{0, 1\}^n$ which satisfy the edge condition, depending on the value of $|F(B^c)|-|F(A^c)|\in\{-1, 0, +1\}$, we define $w_A^B$ to be the edge map $v$ introduced in Definition~\ref{Def of $V$-cube}. Let us call it the \emph{$W$-cube} for short.
\end{definition}

Similar to Lemma \ref{lemma4.8}, one can show that the cube $(W, w)$ is also commutative.

\begin{example}
The $W$-cube of the ribbon graph indexed by $11$ in Figure \ref{figure1} can be found in Figure \ref{figure6}.
\end{example}

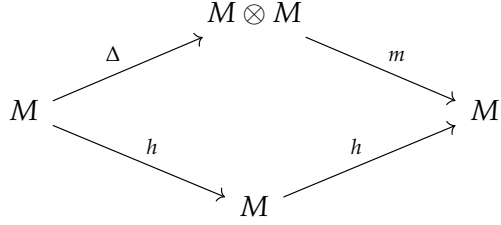
\begin{figure}[htbp]
\centering
\begin{tikzcd}
& & M\otimes M \arrow[rrd, "m"] & & \\
M \arrow[rrd, "h"] \arrow[rru, "\Delta"] & & & & M \\
& & M \arrow[rru, "h"] & &
\end{tikzcd}
\caption{A $W$-cube}
\label{figure6}
\end{figure}

\begin{remark}
In any Frobenius algebra $(A,m,\Delta,\epsilon)$ on a coefficient ring $R$, the map
$$A\otimes A\overset{m}{\longrightarrow} A\overset{\epsilon}{\to}R$$
induces a self-dual isomorphism $A^*\cong A$. While for $M$, this isomorphism maps $x^i$ to $x^{2-i}$ for $i\in\{0, 1, 2\}$. And it is easy to check that the multiplication $m$ and the comultiplication $\Delta$ are dual to each other in the meaning of this isomorphism, as well as the trace map $\epsilon$ and the unit map $u$. Hence the $W$-cube actually can be seen as the dual cube of the $V$-cube. While the conception of dual cube is of less concern in this paper so we omit the accurate illustration here.
\end{remark}

\section{Cochain complex}\label{section5}
In this section, we introduce a cochain complex such that the partial-dual genus polynomial $^\partial\varepsilon_G(z)$ can be recovered from the graded Euler characteristic of this cochain complex. Actually, the polynomial being categorified is not the partial-dual genus polynomial, but the graded partial-dual genus polynomial. More precisely, the graded Euler characteristic of this cochain complex equals the 2-variable polynomial $\tilde{e}_G(w, z)$. In order to do this, for any ribbon graph $G$, we construct an anti-commutative cube $(Cube(G),d)$ by taking the tensor product of its $S$-, $F$-, $V$-, $W$-cubes together. We name it the \emph{partial dual cube} of the ribbon graph $G$. Based on this cube, we construct a cochain complex. It will be found that its cohomology groups categorify the polynomial $\tilde{e}_G(w, z)$.

\begin{example}
The partial dual cube of the graph indexed by $11$ in Figure \ref{figure1} is shown in Figure \ref{figure7}.
\end{example}

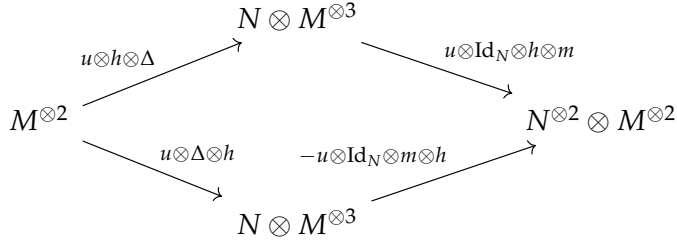
\begin{figure}[htbp]
\centering
\begin{tikzcd}
    &  & N\otimes M^{\otimes 3}\arrow[rrd, "u\otimes\mathrm{Id}_N\otimes h\otimes m"] &  &   \\
M^{\otimes 2}\arrow[rrd, "u\otimes\Delta\otimes h"] \arrow[rru, "u\otimes h\otimes\Delta"] &  & &  & N^{\otimes 2}\otimes M^{\otimes 2} \\
    &  & N\otimes M^{\otimes 3} \arrow[rru, "-u\otimes \mathrm{Id}_N\otimes m\otimes h"] &  &
\end{tikzcd}
\caption{The partial dual cube of the rightmost ribbon graph in Figure \ref{figure1}}
\label{figure7}
\end{figure}

The grading structures of $M$ and $N$ induce a bi-grading structure for each vertex space of $(Cube(G), d)$. That is, for each monoid in a vertex space, its bi-degree is a couple of integers, where the first (second, resp.) is the degree induced by the elements in $N$ ($M$, resp.). For example, the bi-degree of $1_N\otimes y\otimes 1_M\otimes x\otimes x^2 \in N^{\otimes 2}\otimes M^{\otimes 3}$ is $(0+1,1+0+(-1))=(1,0)$, where $1_N$ and $1_M$ are the identity elements in $N$ and $M$ respectively.

A $\mathbb{Z}[\sqrt{3}]\oplus \mathbb{Z}[\sqrt{3}]$-graded or a bi-graded $\mathbb{Z}[\sqrt{3}]$-module is a $\mathbb{Z}[\sqrt{3}]$-module $\mathcal{M}$ with a family of submodules $\mathcal{M}_{m, k}$, where $m, k\in\mathbb{Z}$, such that $\mathcal{M}=\oplus_{m, k\in\mathbb{Z}}\mathcal{M}_{m, k}$, and the elements of $\mathcal{M}_{m, k}$ are homogeneous elements of bi-degree $(m, k)$. The \emph{bi-graded dimension} (or \emph{bi-quantum dimension}) of $\mathcal{M}$ over $\mathbb{Z}[\sqrt{3}]$ is a 2-variable power series
$$\mathrm{qdim} M=\sum_{m, k\in\mathbb{Z}}p^mq^k\cdot\mathrm{dim}_\mathbb{R} (\mathcal{M}_{m, k}\otimes\mathbb{R}).$$

Following \cite{Kho2000}, we use the notation $\{l\}$ $(l\in\mathbb{Z})$ to denote the shift operator on the second grading, meaning that to decrease the second grading of each monoid by $l$. In other words, we have
$$\mathcal{M}\{l\}_{m, k}=\mathcal{M}_{m, k+l},$$
where $\mathcal{M}$ is a bi-graded $\mathbb{Z}[\sqrt{3}]$-module. Now we can define our cochain complex.

\begin{definition}
For a given ribbon graph $G$, we define the \emph{partial dual cochain complex} to be $(C(G), \partial)$, where the cochain groups are defined as
\begin{center}
$C^i(G)=\mathop{\oplus}_{|A|=i}Cube(G)_A\{-2i\}$,
\end{center}
and the coboundary map $\partial^i: C^i(G)\to C^{i+1}(G)$ is given by $\partial^i=\sum d_A^B$, here the sum is taken over all pairs $\{A, B\}$ such that $|A|=i, |B|=i+1$ and $A, B$ satisfy the edge condition.
\end{definition}

\begin{theorem}\label{main theorem}
For an arbitrary ribbon graph $G$, $(C(G),\partial)$ is a bi-graded cochain complex and the cohomology groups $H^*(G)$ are well-defined ribbon graph invariants. In particular, the graded Euler characteristic of the cohomology groups $H^*(G)$ is equal to $\tilde{e}_G(w, z)$ evaluated at $w=-q^2-pq^2$ and $z=(q^{-1}+1+q)^{-1}$.
\end{theorem}
\begin{proof}
The reason why the coboundary maps satisfy $\partial^{i+1}\circ\partial^i=0$ is that $(Cube(G),d)$ is anti-commutative, as it is the tensor product of one anti-commutative cube and three commutative cubes. On the other hand, since $m,\Delta, h$ are all of degree $(0,-1)$ and $f$ is of degree $(0,0)$, thanks to the shift operator $\{-2i\}$, now the differential $\partial$ preserves the bi-degree. Hence $(C(G),\partial)$ is a bi-graded cochain complex. This is to say, decomposing elements by bi-degree yields
\begin{center}
$(C(G),\partial)=\mathop{\oplus}_{m, k\in\mathbb{Z}}(C(G)_{m, k},\partial_{m, k})$,
\end{center}
where $C(G)_{m, k}$ is the submodules of $C(G)$ generated by monoids with degree $(m, k)$ and $\partial_{m, k}$ is the restriction of $\partial$ on $C(G)_{m, k}$. Furthermore, if we use $H^*_{m, k}(G)$ to denote the cohomology groups of $(C(G)_{m, k},\partial_{m, k})$, then we have
\begin{center}
$H^*(G)=\mathop{\oplus}_{m, k\in\mathbb{Z}}H^*_{m, k}(G)$.
\end{center}

Since the definition of the cochain complex needs a fixed order of the edges of $G$, in order to show the cohomology groups $H^*(G)$ are well-defined, we need to prove that $H^*(G)$ actually are independent of the choice of this ordering.

Assume $E(G)=\{e_1, \cdots, e_i, e_{i+1}, \cdots, e_n\}$ and $G'$ is the same ribbon graph as $G$ but the edges are reordered as $E(G')=\{e_1, \cdots, e_{i+1}, e_i, \cdots, e_n\}$, i. e. the positions of $e_i$ and $e_{i+1}$ are switched. It is sufficient to show that $H^*(G)$ and $H^*(G')$ are isomorphic, since any other order can be obtained from the standard order via finitely many this kind of permutations. Since $C^i(G)=\mathop{\oplus}_{|A|=i}Cube(G)_A\{-2i\}$, it is enough to define an isomorphism $f_A$ restricted on each submodule $Cube(G)_A\{-2i\}$. Consider the map $f_A: Cube(G)_A\{-2i\}\to Cube(G')_A\{-2i\}$, which is defined as
\begin{center}
$f_A=\left\{
\begin{aligned}
-id &, \text{ if } \{e_i, e_{i+1}\}\subseteq A,\\
id &, \text{ otherwise.}
\end{aligned}
\right.$
\end{center}
Now we define the map $f: C^i(G)\to C^i(G')$ as $f=\mathop{\oplus}_{A\subseteq E(G), |A|=i}f_A$. It is a routine exercise to check that $f$ is a cochain map which induces the desired isomorphism from $H^i(G)$ to $H^i(G')$.

For the graded Euler characteristic of $H^*(G)$, since the differential is bi-degree preserving and all cochain groups are finite dimensional, one observes that the graded (quantum) Euler characteristic of $(C(G),\partial)$ equals
\begin{align*}
\chi_q(C)&=\sum_{i=0}^{|E(G)|}(-1)^i\mathrm{qdim} H^i(G)\\
  &=\sum_{i=0}^{|E(G)|}(-1)^i\mathrm{qdim} C^i(G)\\
  &=\sum_{i=0}^{|E(G)|}(-1)^i\sum_{|A|=i}\mathrm{qdim} Cube(G)_A\{-2i\}\\
  &=\sum_{A\subseteq E(G)}(-1)^{|A|}q^{2|A|}\mathrm{qdim} (S_A\otimes F_A\otimes V_A\otimes W_A)\\
  &=\sum_{A\subseteq E(G)}(-q^2)^{|A|}(\mathrm{qdim} N)^{|A|}(\mathrm{qdim} M)^{|F(A)|+|F(A^c)|}\\
  &=\sum_{A\subseteq E(G)}(-q^2)^{|A|}(1+p)^{|A|}(q^{-1}+1+q)^{|F(A)|+|F(A^c)|}\\
  &=\sum_{A\subseteq E(G)}(-q^2-pq^2)^{|A|}(q^{-1}+1+q)^{|F(A)|+|F(A^c)|}\\
  &\xlongequal{w=-q^2-pq^2, z=(q^{-1}+1+q)^{-1}}\sum_{A\subseteq E(G)}w^{|A|}z^{-|F(A)|-|F(A^c)|}\\
  &=\tilde{e}_G(w, z).
\end{align*}
\end{proof}

\begin{corollary}
For a given ribbon graph $G$, both the graded partial-dual genus polynomial and the partial-dual genus polynomial can be recovered from the graded Euler characteristic of the cohomology groups $H^*(G)$.
\end{corollary}
\begin{proof}
According to Theorem \ref{main theorem}, we have
\begin{center}
$^\partial\tilde{\varepsilon}_G(w, z)=z^{2c(G)+|E(G)|}\chi_q(C)\Big{|}_{-q^2-pq^2=w, q^{-1}+1+q=z^{-1}}$,
\end{center}
and
\begin{center}
$^\partial\varepsilon_G(z)=z^{2c(G)+|E(G)|}\chi_q(C)\Big{|}_{-q^2-pq^2=1,q^{-1}+1+q=z^{-1}}$.
\end{center}
\end{proof}

\section{An example}\label{section6}
In this section, we compute the cohomology groups of the rightmost ribbon graph shown in Figure \ref{figure1}. The cochain complex is as the bottom row of Figure~\ref{A cochain complex}, where the other cochain groups vanish. On the other hand, given that this complex is bigraded, we use the subscript $(j,k)$ to represent corresponding objects with bi-degree $(j, k)$. Additionally, with the degree shift, these cohomology groups is nonzero only with the first degree $j$ being $0,1,2$ and the second degree $k$ being $-2,-1,\cdots,6$. Hence we can calculate the cohomology groups of these $3\times 9=27$ subcomplexes one by one and the aim cohomology is the direct sum of them.

\begin{figure}[htbp]
  \centering
  \begin{tikzcd}
    & &  & N\otimes M^{\otimes 3}\{-2\}\arrow[rrd, "-u\otimes\mathrm{Id}_N\otimes h\otimes m"] & & \\
& M^{\otimes 2}\arrow[rrd, "u\otimes\Delta\otimes h"] \arrow[rru, "u\otimes h\otimes\Delta"]\arrow[dd,equal] & & \oplus & & N^{\otimes 2}\otimes M^{\otimes 2}\{-4\}\arrow[dd,equal]\\
    & &  & N\otimes M^{\otimes 3}\{-2\} \arrow[rru, "u\otimes \mathrm{Id}_N\otimes m\otimes h"]\arrow[d,equal] & & \\
    (C,\partial): & C^0\arrow[rr,"\partial^0"] & & C^1\arrow[rr,"\partial^1"] & & C^2
\end{tikzcd}
\caption{A cochain complex}
\label{A cochain complex}
\end{figure}

Due to the special structure of the map $f$, the cases with the first degree equaling to $0$ is comparatively intricate. We list our calculation process as follows, where the subscript of unit element in $M$ is omitted for short.

\begin{itemize}
  \item $(j,k)=(0,-2)$:
  \begin{align*}
    C^0_{(0,-2)}&=\langle x^2\otimes x^2\rangle=\mathbb{Z}[\sqrt{3}], & C^1_{(0,-2)}=0, \quad\quad\quad  & C^2_{(0,-2)}=0;\\
    H^0_{(0,-2)}&=\mathbb{Z}[\sqrt{3}], & H^1_{(0,-2)}=0,\quad\quad\quad & H^2_{(0,-2)}=0.
  \end{align*}
  \item $(j,k)=(0,-1)$:
  \begin{align*}
    C^0_{(0,-1)}&=\langle x\otimes x^2,x^2\otimes x \rangle=\mathop\oplus_2\mathbb{Z}[\sqrt{3}], \\
    C^1_{(0,-1)}&=\langle
        \biggl(\begin{array}
        {c}
        1_N\otimes x^2\otimes x^2\otimes x^2\\
        0
        \end{array}\biggr),
        \biggl(\begin{array}
        {c}
        0\\
        1_N\otimes x^2\otimes x^2\otimes x^2
        \end{array}\biggr) \rangle=\mathop\oplus_2\mathbb{Z}[\sqrt{3}],\\
    C^2_{(0,-1)}&=0;
  \end{align*}
  \begin{equation*}
    \partial^0: x\otimes x^2\mapsto\sqrt{3}\biggl(\begin{array}
      {c}
      1_N\otimes x^2\otimes x^2\otimes x^2\\
      0
      \end{array}\biggr),\,x^2\otimes x\mapsto\sqrt{3}\biggl(\begin{array}
        {c}
        0\\
        1_N\otimes x^2\otimes x^2\otimes x^2
        \end{array}\biggr);
  \end{equation*}
  \begin{equation*}
    H^0_{(0,-1)}=0,\,\quad H^1_{(0,-1)}=\mathop\oplus_2\mathbb{Z}[\sqrt{3}]/\sqrt{3}\mathbb{Z}[\sqrt{3}]=\mathop\oplus_2\mathbb{Z}_3,\,\quad H^2_{(0,-1)}=0.
  \end{equation*}
  \item $(j,k)=(0,0)$:
  \begin{align*}
    C^0_{(0,0)}&=\langle 1\otimes x^2,x\otimes x,x^2\otimes 1\rangle=\mathop\oplus_3\mathbb{Z}[\sqrt{3}],\\
    C^1_{(0,0)}&=\langle \biggl(\begin{array}
      {c}
      1_N\otimes x\otimes x^2\otimes x^2\\
      0
      \end{array}\biggr),
      \biggl(\begin{array}
        {c}
        1_N\otimes x^2\otimes x\otimes x^2\\
        0
        \end{array}\biggr) ,
        \biggl(\begin{array}
          {c}
          1_N\otimes x^2\otimes x^2\otimes x\\
          0
          \end{array}\biggr) ,\\
      &\quad \biggl(\begin{array}
      {c}
      0\\
      1_N\otimes x\otimes x^2\otimes x^2
      \end{array}\biggr),
      \biggl(\begin{array}
        {c}
        0\\
        1_N\otimes x^2\otimes x\otimes x^2
        \end{array}\biggr),
        \biggl(\begin{array}
          {c}
          0\\
          1_N\otimes x^2\otimes x^2\otimes x
          \end{array}\biggr) \rangle\\
    &=\mathop\oplus_6\mathbb{Z}[\sqrt{3}],\\
    C^2_{(0,0)}&=0;
  \end{align*}
  \begin{align*}
    \partial^0:\,1\otimes x^2&\mapsto\sqrt{3}\biggl(\begin{array}
      {c}
      1_N\otimes x\otimes x^2\otimes x^2\\
      0
      \end{array}\biggr),\\
      x\otimes x&\mapsto\sqrt{3}\biggl(\begin{array}
        {c}
        1_N\otimes x^2\otimes(x\otimes x^2+x^2\otimes x)\\
        1_N\otimes (x\otimes x^2+x^2\otimes x)\otimes x^2
        \end{array}\biggr),\\
        x^2\otimes 1&\mapsto\sqrt{3}\biggl(\begin{array}
          {c}
          0\\
          1_N\otimes x^2\otimes x^2\otimes x
          \end{array}\biggr);
  \end{align*}
  \begin{equation*}
    H^0_{(0,0)}=0,\quad H^1_{(0,0)}=(\mathop\oplus_3\mathbb{Z}[\sqrt{3}])\oplus(\mathop\oplus_3\mathbb{Z}_3),\quad H^2_{(0,0)}=0.
  \end{equation*}
  \item $(j,k)=(0,1)$:
  \begin{align*}
    C^0_{(0,1)}&=\langle 1\otimes x,x\otimes 1 \rangle=\mathop\oplus_2\mathbb{Z}[\sqrt{3}],\\
    C^1_{(0,1)}&=\langle \biggl(\begin{array}
      {c}
      1_N\otimes 1\otimes x^2\otimes x^2\\
      0
      \end{array}\biggr),
      \biggl(\begin{array}
        {c}
        1_N\otimes x^2\otimes 1\otimes x^2\\
        0
        \end{array}\biggr) ,
        \biggl(\begin{array}
          {c}
          1_N\otimes x^2\otimes x^2\otimes 1\\
          0
          \end{array}\biggr) ,\\
      &\biggl(\begin{array}
        {c}
        1_N\otimes x^2\otimes x\otimes x\\
        0
        \end{array}\biggr),
        \biggl(\begin{array}
          {c}
          1_N\otimes x\otimes x^2\otimes x\\
          0
          \end{array}\biggr) ,
          \biggl(\begin{array}
            {c}
            1_N\otimes x\otimes x\otimes x^2\\
            0
            \end{array}\biggr) ,\cdots\rangle\\
            &=\mathop\oplus_{12}\mathbb{Z}[\sqrt{3}],\\
    C^2_{(0,1)}&=0;
  \end{align*}
  where the ellipsis means to repeat the elements before it with the two rows switched;
  \begin{align*}
    \partial^0:\,1\otimes x &\mapsto\sqrt{3}\biggl(
      \begin{array}{c}
        1_N\otimes x\otimes (x^2\otimes x+x\otimes x^2)\\
        1_N\otimes (1\otimes x^2+x\otimes x+x^2\otimes 1)\otimes x^2
      \end{array}\biggr),\\
      x\otimes 1 &\mapsto \sqrt{3}\biggl(
      \begin{array}{c}
        1_N\otimes x^2\otimes(1\otimes x^2+x\otimes x+x^2\otimes 1)\\
        1_N\otimes (x\otimes x^2+x^2\otimes x)\otimes x
      \end{array}\biggr);
  \end{align*}
  \begin{equation*}
    H^0_{(0,1)}=0,\quad H^1_{(0,1)}=(\mathop\oplus_{10}\mathbb{Z}[\sqrt{3}])\oplus(\mathop\oplus_2\mathbb{Z}_3),\quad H^2_{(0,1)}=0.
  \end{equation*}
  \item $(j,k)=(0,2)$:
  \begin{align*}
    C^0_{(0,2)}&=\langle 1\otimes 1 \rangle =\mathbb{Z}[\sqrt{3}],\\
    C^1_{(0,2)}&=\langle \biggl(\begin{array}
      {c}
      1_N\otimes 1\otimes x\otimes x^2\\
      0
      \end{array}\biggr),
      \biggl(\begin{array}
        {c}
        1_N\otimes x\otimes 1\otimes x^2\\
        0
        \end{array}\biggr) ,
        \biggl(\begin{array}
          {c}
          1_N\otimes x^2\otimes 1\otimes x\\
          0
          \end{array}\biggr) ,\\
      &\quad\biggl(\begin{array}
        {c}
        1_N\otimes 1\otimes x^2\otimes x\\
        0
        \end{array}\biggr),
        \biggl(\begin{array}
          {c}
          1_N\otimes x\otimes x^2\otimes 1\\
          0
          \end{array}\biggr) ,
          \biggl(\begin{array}
            {c}
            1_N\otimes x^2\otimes x\otimes 1\\
            0
            \end{array}\biggr) ,\\
            &\quad\biggl(\begin{array}
              {c}
              1_N\otimes x\otimes x\otimes x\\
              0
              \end{array}\biggr) ,
             \cdots\rangle=\mathop\oplus_{14}\mathbb{Z}[\sqrt{3}],\\
    C^2_{(0,2)}&=\langle 1_N\otimes 1_N\otimes x^2\otimes x^2\rangle =\mathbb{Z}[\sqrt{3}];
  \end{align*}
  \begin{equation*}
    H^0_{(0,2)}=0,\quad H^1_{(0,2)}=(\mathop\oplus_{12}\mathbb{Z}[\sqrt{3}])\oplus \mathbb{Z}_3,\quad H^2_{(0,2)}=\mathbb{Z}_3.
  \end{equation*}
  \item $(j,k)=(0,3)$:
  \begin{align*}
    C^0_{(0,3)}&=0,\\
    C^1_{(0,3)}&=\langle \biggl(\begin{array}
      {c}
      1_N\otimes 1\otimes x\otimes x\\
      0
      \end{array}\biggr),
      \biggl(\begin{array}
        {c}
        1_N\otimes x\otimes 1\otimes x\\
        0
        \end{array}\biggr) ,
        \biggl(\begin{array}
          {c}
          1_N\otimes x\otimes x\otimes 1\\
          0
          \end{array}\biggr) ,\\
      &\biggl(\begin{array}
        {c}
        1_N\otimes x^2\otimes 1\otimes 1\\
        0
        \end{array}\biggr),
        \biggl(\begin{array}
          {c}
          1_N\otimes 1\otimes x^2\otimes 1\\
          0
          \end{array}\biggr) ,
          \biggl(\begin{array}
            {c}
            1_N\otimes 1\otimes 1\otimes x^2\\
            0
            \end{array}\biggr) ,\cdots\rangle\\
            &=\mathop\oplus_{12}\mathbb{Z}[\sqrt{3}],\\
    C^3_{(0,3)}&=\langle 1_N\otimes 1_N\otimes x\otimes x^2, 1_N\otimes 1_N\otimes x^2\otimes x\rangle=\mathop\oplus_2\mathbb{Z};
  \end{align*}
  \begin{equation*}
    H^0_{(0,3)}=0,\quad H^1_{(0,3)}=\mathop\oplus_{10}\mathbb{Z}[\sqrt{3}],\quad H^3_{(0,3)}=\mathop\oplus_2\mathbb{Z}_3.
  \end{equation*}
  \item $(j,k)=(0,4)$:
  \begin{align*}
    C^0_{(0,4)}&=0,\\
    C^1_{(0,4)}&=\langle\biggl(
      \begin{array}{c}
        1_N\otimes 1\otimes 1\otimes x\\
        0
      \end{array}\biggr),\biggl(
        \begin{array}{c}
          1_N\otimes 1\otimes x\otimes 1\\
          0
        \end{array}\biggr),\biggl(
          \begin{array}{c}
            1_N\otimes x\otimes 1\otimes 1\\
            0
          \end{array}\biggr),\cdots\rangle\\
          &=\mathop\oplus_6\mathbb{Z}[\sqrt{3}],\\
    C^2_{(0,4)}&=\langle 1_N\otimes 1_N\otimes x\otimes x,1_N\otimes 1_N\otimes 1\otimes x^2,1_N\otimes 1_N\otimes x^2\otimes 1\rangle=\mathop\oplus_3\mathbb{Z}[\sqrt{3}];
  \end{align*}
  \begin{equation*}
    H^0_{(0,4)}=0,\quad H^1_{(0,4)}=\mathop\oplus_3\mathbb{Z}[\sqrt{3}],\quad H^2_{(0,4)}=\mathop\oplus_3\mathbb{Z}_3.
  \end{equation*}
  \item $(j,k)=(0,5)$:
  \begin{align*}
    C^0_{(0,5)}&=0,\\
    C^1_{(0,5)}&=\langle \biggl(
      \begin{array}{c}
        1_N\otimes 1\otimes 1\otimes 1\\
        0
      \end{array}\biggr) ,\biggl(
        \begin{array}{c}
          0\\
          1_N\otimes 1\otimes 1\otimes 1
        \end{array}\biggr)\rangle=\mathop\oplus_2\mathbb{Z}[\sqrt{3}] ,\\
    C^2_{(0,5)}&=\langle 1_N\otimes 1_N\otimes 1\otimes x,1_N\otimes 1_N\otimes x\otimes 1 \rangle=\mathop\oplus_2\mathbb{Z}[\sqrt{3}].
  \end{align*}
  \begin{equation*}
    H^0_{(0,5)}=0,\quad H^1_{(0,5)}=0,\quad H^2_{(0,5)}=\mathop\oplus_2\mathbb{Z}_3.
  \end{equation*}
  \item $(j,k)=(0,6)$:
  \begin{align*}
    C^0_{(0,6)}=0,\quad C^1_{(0,6)}=0,\quad C^2_{(0,6)}&=\langle 1_N\otimes 1_N\otimes 1\otimes 1\rangle=\mathbb{Z}[\sqrt{3}];\\
    H^0_{(0,6)}=0,\quad H^1_{(0,6)}=0,\quad H^2_{(0,6)}&=\mathbb{Z}[\sqrt{3}].
  \end{align*}
\end{itemize}
When the first degree is more than $1$, we have that
$$C^0_{(\geq 1,*)}=H^0_{(\geq 1,*)}=0,$$
and we will not repeat it in our progress.
\begin{itemize}
  \item $(j,k)=(1,-2)$:
  $$C^1_{(0,-2)}=C^2_{(0,-2)}=0,\quad\quad H^1_{(0,-2)}=H^2_{(0,-2)}=0.$$
  \item $(j,k)=(1,-1)$:
  $$C^1_{(1,-1)}=\mathop\oplus_2\mathbb{Z}[\sqrt{3}],\quad C^2_{(1,-1)}=0;\quad \quad
  H^1_{(1,-1)}=\mathop\oplus_2\mathbb{Z}[\sqrt{3}],\quad H^2_{(1,-1)}=0.$$
  \item $(j,k)=(1,0)$:
  $$
  C^1_{(1,0)}=\mathop\oplus_6\mathbb{Z}[\sqrt{3}],\quad C^2_{(1,0)}=0;\quad\quad
  H^1_{(1,0)}=\mathop\oplus_6\mathbb{Z}[\sqrt{3}],\quad H^2_{(1,0)}=0.$$
  \item $(j,k)=(1,1)$:
  $$C^1_{(1,1)}=\mathop\oplus_{12}\mathbb{Z}[\sqrt{3}],\quad C^2_{(1,1)}=0;\quad \quad
  H^1_{(1,1)}=\mathop\oplus_{12}\mathbb{Z}[\sqrt{3}],\quad H^2_{(1,1)}=0.$$
  \item $(j,k)=(1,2)$:
  \begin{align*}
    C^1_{(1,2)}&=\mathop\oplus_{14}\mathbb{Z}[\sqrt{3}],\quad C^2_{(1,2)}=\langle 1_N\otimes y\otimes x^2\otimes x^2,y\otimes 1_N\otimes x^2\otimes x^2\rangle=\mathop\oplus_{2}\mathbb{Z}[\sqrt{3}] ;\\
    H^1_{(1,2)}&=\mathop\oplus_{13}\mathbb{Z}[\sqrt{3}],\quad H^2_{(1,2)}=\mathbb{Z}[\sqrt{3}]\oplus\mathbb{Z}_3.
  \end{align*}
  \item $(j,k)=(1,3)$:
  \begin{align*}
    C^1_{(1,3)}=\mathop\oplus_{12}\mathbb{Z}[\sqrt{3}],\quad C^2_{(1,3)}&=\mathop\oplus_{4}\mathbb{Z}[\sqrt{3}];\\
    H^1_{(1,3)}=\mathop\oplus_{10}\mathbb{Z}[\sqrt{3}],\quad  H^2_{(1,3)}&=(\mathop\oplus_{2}\mathbb{Z}[\sqrt{3}])\oplus( \mathop\oplus_{2}\mathbb{Z}_3).
  \end{align*}
  \item $(j,k)=(1,4)$:
  \begin{align*}
    C^1_{(1,4)}=\mathop\oplus_{6}\mathbb{Z}[\sqrt{3}],\quad C^2_{(1,4)}&=\mathop\oplus_{6}\mathbb{Z}[\sqrt{3}];\\
    H^1_{(1,4)}=\mathop\oplus_{3}\mathbb{Z}[\sqrt{3}],\quad  H^2_{(1,4)}&=(\mathop\oplus_{3}\mathbb{Z}[\sqrt{3}])\oplus( \mathop\oplus_{3}\mathbb{Z}_3).
  \end{align*}
  \item $(j,k)=(1,5)$:
  \begin{align*}
    C^1_{(1,5)}=\mathop\oplus_{2}\mathbb{Z}[\sqrt{3}],\quad C^2_{(1,5)}&=\mathop\oplus_{4}\mathbb{Z}[\sqrt{3}];\\
    H^1_{(1,5)}=0,\quad  H^2_{(1,5)}&=(\mathop\oplus_{2}\mathbb{Z}[\sqrt{3}])\oplus( \mathop\oplus_{2}\mathbb{Z}_3).
  \end{align*}
  \item $(j,k)=(1,6)$:
  $$C^1_{(1,6)}=0,\quad C^2_{(1,6)}=\mathop\oplus_{2}\mathbb{Z}[\sqrt{3}];\quad \quad H^1_{(1,6)}=0,\quad H^2_{(1,6)}=\mathop\oplus_{2}\mathbb{Z}[\sqrt{3}].$$
\end{itemize}
If the first degree is $2$, the situation is trivial, say
$$C^0_{(2,*)}=C^1_{(2,*)}=H^0_{(2,*)}=H^1_{(2,*)}=0,\quad C^2_{(2,*)}=H^2_{(2,*)}=y\otimes y\otimes M\otimes M\{-4\}.$$
Or specifically speaking, we have
$$H^2_{(2,2)}=H^2_{(2,6)}=\mathbb{Z}[\sqrt{3}],\quad
H^2_{(2,3)}=H^2_{(2,5)}=\mathop\oplus_2\mathbb{Z}[\sqrt{3}],\quad H^2_{(2,4)}=\mathop\oplus_3\mathbb{Z}[\sqrt{3}].$$

The graded Euler characteristic of the homology groups can be written as
\begin{align*}
  \chi_q(C(G))&=\sum_{i,j,k}(-1)^ip^jq^k\mathrm{dim}\,H^i_{(j,k)}(G)\\
  &=p^0(q^{-2}-3-10q-12q^2-10q^3-3q^4+q^6)\\
  &+p^1(-2q^{-1}-6-12q-12q^2-8q^3+2q^5+2q^6)\\
  &+p^2(q^2+2q^3+3q^4+2q^5+q^6)\\
  &\xlongequal{p=-\frac{w}{q^2}-1}q^{-2}+2q^{-1}+3+2q+q^2\\
  &+2w(q^{-3}+3q^{-2}+6q^{-1}+7+6q+3q^2+q^3)\\
  &+w^2(q^{-2}+2q^{-1}+3+2q+q^2)\\
  &\xlongequal{q+1+q^{-1}=z^{-1}}z^{-2}+2wz^{-3}+w^2z^{-2}\\
  &=\tilde{e}_G(w, z)
\end{align*}
It follows that
\begin{center}
$^\partial\tilde\varepsilon_G(w, z)=z^{2c(G)+|E(G)|}\tilde{e}_G(w, z)=z^2+2wz+w^2z^2$
\end{center}
and
\begin{center}
$^\partial\varepsilon_G(z)=^\partial\tilde\varepsilon_G(1, z)=2z^2+2z$,
\end{center}
which coincides with the results we obtained in Example \ref{example 2.3} and Example \ref{example2.5}. On the other hand, one notices that some terms in the cohomology groups cancel out when one calculates the graded Euler characteristic. See $H_{(1, 2)}^1$ and $H_{(1, 2)}^2$ for instance. This means that the cohomology indeed contains more information comparing with the graded partial-dual genus polynomial.

\section{Some applications}\label{section7}
In this section, we review some basic operations on ribbon graphs discussed in \cite{GMT2020}. The behaviors of the partial-dual genus polynomial under these operations are given in \cite{GMT2020} based on the topological definition of the partial-dual genus polynomial. In the first subsection, we first extend these results to the graded partial-dual genus polynomial. Then in the rest subsections, we show that these properties can be categorified in the cohomology theory defined in Section \ref{section5}.

\subsection{Some properties of the graded partial-dual genus polynomial}
In \cite{Mof2013}, Moffatt defines the so-called \emph{ribbon-join} operation on two disjoint ribbon graphs $G_1$ and $G_2$, denoted by $G_1\vee G_2$, in two steps:
\begin{itemize}
  \item Choose an arc $p_1$ on the boundary of a vertex-disk $v_1$ of $G_1$ that lies between two consecutive ribbon ends, and choose another arc $p_2$ on the boundary of a vertex-disk $v_2$ of $G_2$.
  \item Paste vertex-disks $v_1$ and $v_2$ together by identifying the arcs $p_1$ and $p_2$.
\end{itemize}
The following Proposition extends the result of \cite[Proposition 3.2 (a)]{GMT2020} from partial-dual genus polynomial to graded partial-dual genus polynomial.

\begin{proposition}\label{proposition7.1}
  Let $G_1$ and $G_2$ be disjoint ribbon graphs, then $$^\partial\tilde\varepsilon_{G_1\cup G_2}(w, z)= ^\partial\tilde\varepsilon_{G_1\vee G_2}(w, z)=^\partial\tilde\varepsilon_{G_1}(w, z)^\partial\tilde{\varepsilon}_{G_2}(w, z).$$
\end{proposition}
\begin{proof}
According to our combinatorial formula of the partial-dual genus polynomial given in Lemma \ref{combinatorial def}, together with the definition of the graded partial-dual genus polynomial, one computes
  \begin{align*}
    &^\partial\tilde\varepsilon_{G_1\cup G_2}(w, z)\\
    =&z^{2c(G_1\cup G_2)+|E(G_1\cup G_2)|}\sum_{A\subseteq E(G_1\cup G_2)}w^{|A|}z^{-|F_\cup(A)|-|F_\cup(A^c)|}\\
    =&z^{2c(G_1)+2c(G_2)+|E(G_1)|+|E(G_2)|}\sum_{\substack{A_1\subseteq E(G_1)\\A_2\subseteq E(G_2)}}w^{|A_1|+|A_2|}z^{-|F_1(A_1)|-|F_2(A_2)|-|F_1(A_1^c)|-|F_2(A_2^c)|}\\
    =&(z^{2c(G_1)+|E(G_1)|}\sum_{A_1\subseteq E(G_1)}w^{|A_1|}z^{-|F_1(A_1)|-|F_1(A_1^c)|})\\
    \times&(z^{2c(G_2)+|E(G_2)|}\sum_{A_2\subseteq E(G_2)}w^{|A_2|}z^{-|F_2(A_2)|-|F_2(A_2^c)|})\\
    =&^\partial\tilde\varepsilon_{G_1}(w, z)^\partial\tilde\varepsilon_{G_2}(w, z)
    \end{align*}
    and
    \begin{align*}
    &^\partial\tilde\varepsilon_{G_1\vee G_2}(w, z)\\
    =&z^{2c(G_1\vee G_2)+|E(G_1\vee G_2)|}\sum_{A\subseteq E(G_1\vee G_2)}w^{|A|}z^{-|F_\vee(A)|-|F_\vee(A^c)|}\\
    =&z^{2c(G_1)+2c(G_2)-2+|E(G_1)|+|E(G_2)|}\sum_{\substack{A_1\subseteq E(G_1)\\A_2\subseteq E(G_2)}}w^{|A_1|+|A_2|}z^{-|F_1(A_1)|-|F_2(A_2)|+1-|F_1(A_1^c)|-|F_2(A_2^c)|+1}\\
    =&(z^{2c(G_1)+|E(G_1)|}\sum_{A_1\subseteq E(G_1)}w^{|A_1|}z^{-|F_1(A_1)|-|F_1(A_1^c)|})\\
    \times&(z^{2c(G_2)+|E(G_2)|}\sum_{A_2\subseteq E(G_2)}w^{|A_2|}z^{-|F_2(A_2)|-|F_2(A_2^c)|})\\
    =&^\partial\tilde\varepsilon_{G_1}(w, z)^\partial\tilde\varepsilon_{G_2}(w, z),
  \end{align*}
where $F_{\cup}, F_{\vee}, F_1$ and $F_2$ counts the number of face-disks of the subgraphs of $G_1\cup G_2, G_1\vee G_2, G_1$ and $ G_2$ respectively.
\end{proof}

To construct a \emph{bar-amalgamation} of two disjoint ribbon graphs $G_1$ and $G_2$, denoted by $G_1\eqcirc G_2$, the authors of \cite{GMT2020} begin as the first step of the ribbon-join operation by selecting arcs $p_1$ and $p_2$, on the boundaries of vertex-disks $v_1$ and $v_2$ of $G_1$ and $G_2$, and then paste one end of a new ribbon to $p_1$ and the other end to $p_2$.

\begin{proposition}\label{proposition7.2}
Let $G_1$ and $G_2$ be disjoint ribbon graphs, then
\begin{center}
$^\partial\tilde\varepsilon_{G_1\eqcirc G_2}(w, z)=^\partial\tilde\varepsilon_{G_1}(w, z)^\partial\tilde\varepsilon_{G_2}(w, z)+w^\partial\tilde\varepsilon_{G_1}(w, z)^\partial\tilde\varepsilon_{G_2}(w, z).$
\end{center}
\end{proposition}
\begin{proof}
Similar to the proof of Proposition \ref{proposition7.1}, let us use $F_\eqcirc$ to denote the number of face-disks of the subgraphs of $G_1\eqcirc G_2$. Then we have
  \begin{align*}
    &^\partial\tilde\varepsilon_{G_1\eqcirc G_2}(w, z)\\
    =&z^{2c(G_1\eqcirc G_2)+|E(G_1\eqcirc G_2)|}\sum_{A\subseteq E(G_1\eqcirc G_2)}w^{|A|}z^{-|F_\eqcirc(A)|-|F_\eqcirc(A^c)|}\\
    =&z^{2c(G_1)+2c(G_2)-2+|E(G_1|+|E(G_2)|+1}(\sum_{e\notin A\subseteq E(G_1\eqcirc G_2)}+\sum_{e\in A\subseteq E(G_1\eqcirc G_2)})w^{|A|}z^{-|F_\eqcirc(A)|-|F_\eqcirc(A^c)|}\\
    =&z^{2c(G_1)+2c(G_2)+|E(G_1)|+|E(G_2)|-1}(\sum_{\substack{A_1\in E(G_1)\\A_2\in E(G_2)}}w^{|A_1|+|A_2|}z^{-|F_1(A_1)|-|F_2(A_2)|-|F_1(A_1^c)|-|F_2(A_2^c)|+1}\\
    +&\sum_{\substack{A_1\in E(G_1)\\A_2\in E(G_2)}}w^{|A_1|+|A_2|+1}z^{-|F_1(A_1)|-|F_2(A_2)|-|F_1(A_1^c)|-|F_2(A_2^c)|+1})\\
    =&^\partial\tilde\varepsilon_{G_1}(w, z)^\partial\tilde\varepsilon_{G_2}(w, z)+w^\partial\tilde\varepsilon_{G_1}(w, z)^\partial\tilde\varepsilon_{G_2}(w, z)
  \end{align*}
where $e$ denotes the ribbon-edge added to connect $G_1$ and $G_2$.
\end{proof}

By setting $w=1$, then we obtain
\begin{center}
$^\partial\varepsilon_{G_1\eqcirc G_2}(z)=2^\partial\varepsilon_{G_1}(z)^\partial\varepsilon_{G_2}(z)$,
\end{center}
this recovers the result of \cite[Proposition 3.2 (b)]{GMT2020}.

\begin{remark}
  These notations hide the information of the arcs and vertex-disks we choose, while it can be seen from these two properties that the graded partial-dual genus polynomial does not depend on the choices of these arcs and vertex-disks. Actually, sliding an edge-ribbon of a ribbon graph along the boundary from one vertex-disk to another one preserves the partial-dual genus polynomial \cite[Lemma 2.2]{Chm2023}. One can similarly prove that the graded partial-dual genus polynomial is also preserved under edge sliding. Based on this fact, it is easy to observe that the graded partial-dual genus polynomial of the bar-amalgamation of two connected ribbon graphs does not depend on the choices of the arc and vertex-disks. On the aspect of categorification, it is routine to check that the S-, F-, V-, W-cube also do not depend on these choices. In other words, different choices of the arcs and vertex-disks give rise to isomorphic cohomology groups. In particular, by ordering the sliding edge as the last one, it is not difficult to observe that the cohomology groups are also preserved under edge sliding.
\end{remark}

As we mentioned before, the classical Euler-Poincar\'{e} dual of a ribbon graph $G$ has the same genus as $G$. For the generating function $^\partial\varepsilon_{G}(z)$, it is obvious to conclude from Definition \ref{definition2.1} that the partial-dual genus polynomial is preserved under the partial dual with respect to any subset $A\subseteq E(G)$. In other words, for arbitrary $A\subseteq E(G)$ we have $^\partial\varepsilon_{G^A}(z)=^\partial\varepsilon_{G}(z)$. Next we reprove this result by using Lemma \ref{combinatorial def}, since this method will be used again later. The following lemma can be proved directly by using the fact $F(A)=V(G^A)$, but we still want to present an pure geometrical proof here.

\begin{lemma}\label{lemma7.4}
Let $G$ be a ribbon graph and $e\in E(G)$, for arbitrary $e\notin B\subseteq E(G)$, we have
\begin{center}
$F_{G^{\{e\}}}(B)=F_G(B\cup\{e\})$ and $F_{G^{\{e\}}}(B\cup\{e\})=F_G(B)$,
\end{center}
where $F_{G^{\{e\}}}(B)$ denotes the set of the face disks of the ribbon graph obtained from $G^{\{e\}}$ by removing all the edges not in $B$.
\end{lemma}

\begin{proof}
Consider the edge $e$ which connects two disks in Figure \ref{figure0}. Since $e\notin B$, one observes that the ribbon graph obtained from $G^{\{e\}}$ by removing all the edges not in $B$ is homeomorphic to the ribbon graph obtained from $G$ by removing all the edges not in $B\cup\{e\}$. It follows that $F_{G^{\{e\}}}(B)=F_G(B\cup\{e\})$. For the cases that $e$ connects a disk to itself, either twisted or untwisted, one can also prove that there exists a homeomorphism between these two ribbon graphs.

By replacing $G$ with $G^{\{e\}}$, one obtains the second equality.
\end{proof}

\begin{proposition}\label{proposition7.4}
For any ribbon graph $G$ and and any subset $A\subseteq E(G)$, we have
\begin{center}
$^\partial\varepsilon_G(z)=^\partial\varepsilon_{G^A}(z).$
\end{center}
\end{proposition}
\begin{proof}
It suffices to verify the case that $A$ contains exactly one edge. For the general case, it can be proved inductively by using the fact $G^{A\cup\{e\}}=(G^A)^{\{e\}}$.

Assume $A=\{e\}$, first we notice that
\begin{center}
$c(G^{\{e\}})=c(G)$ and $|E(G^{\{e\}})|=|E(G)|.$
\end{center}

By using the natural one-to-one correspondence between the elements in $E(G^{\{e\}})$ and $E(G)$, we set $E(G^{\{e\}})=E(G)=E_1\cup E_2$, where $E_1=\{A\subseteq E(G)|e\in A\}$ and $E_2=\{A\subseteq E(G)|e\notin A\}$. Then one computes
  \begin{align*}
    &^\partial\varepsilon_{G^{\{e\}}}(z)\\
    =&z^{2c(G^{\{e\}})+|E(G^{\{e\}})|}\sum_{B\subseteq E(G^{\{e\}})}z^{-|F_{G^{\{e\}}}(B)|-|F_{G^{\{e\}}}(B^c)|}\\
    =&z^{2c(G)+|E(G)|}(\sum_{B\in E_1}+\sum_{B\in E_2})z^{-|F_{G^{\{e\}}}(B)|-|F_{G^{\{e\}}}(B^c)|}\\
    =&z^{2c(G)+|E(G)|}(\sum_{B\in E_1}z^{-|F_G(B\setminus\{e\})|-|F_G(B^c\cup\{e\})|}+\sum_{B\in E_2}z^{-|F_G(B\cup\{e\})|-|F_G(B^c\setminus \{e\})|})\\
    =&z^{2c(G)+|E(G)|}(\sum_{B\in E_2}z^{-|F_G(B)|-|F_G(B^c)|}+\sum_{B\in E_1}z^{-|F_G(B)|-|F_G(B^c)|})\\
    =&^\partial\varepsilon_{G}(z).
  \end{align*}
\end{proof}

It is worthy to mention that the graded partial-dual genus polynomial $^\partial\tilde\varepsilon_G(w, z)$ is not preserved under the partial dual in general. More precisely, consider a ribbon graph $G$ and an edge $e\in E(G)$, let us still denote $E_1=\{A\subseteq E(G)|e\in A\}$ and $E_2=\{A\subseteq E(G)|e\notin A\}$. Then we have
\begin{center}
$^\partial\tilde\varepsilon_{G}(w, z)=\sum\limits_{A\in E_1}w^{|A|}z^{\varepsilon(G^A)}+\sum\limits_{A\in E_2}w^{|A|}z^{\varepsilon(G^A)}$.
\end{center}
For convenience, we denote $\sum\limits_{A\in E_1}w^{|A|}z^{\varepsilon(G^A)}$ and $\sum\limits_{A\in E_2}w^{|A|}z^{\varepsilon(G^A)}$ by $u(w, z)$ and $v(w, z)$ respectively. In other words, $^\partial\tilde\varepsilon_{G}(w, z)=u(w, z)+v(w, z)$. As a mimic of the proof of Proposition \ref{proposition7.4}, one computes
\begin{align*}
    &^\partial\tilde\varepsilon_{G^{\{e\}}}(w, z)\\
    =&z^{2c(G^{\{e\}})+|E(G^{\{e\}})|}\sum_{A\subseteq E(G^{\{e\}})}w^{|A|}z^{-|F_{G^{\{e\}}}(A)|-|F_{G^{\{e\}}}(A^c)|}\\
    =&z^{2c(G)+|E(G)|}(\sum_{A\in E_1}+\sum_{A\in E_2})w^{|A|}z^{-|F_{G^{\{e\}}}(A)|-|F_{G^{\{e\}}}(A^c)|}\\
    =&z^{2c(G)+|E(G)|}(\sum_{A\in E_1}w^{|A|}z^{-|F_G(A\setminus\{e\})|-|F_G(A^c\cup\{e\})|}+\sum_{A\in E_2}w^{|A|}z^{-|F_G(A\cup\{e\})|-|F_G(A^c\setminus \{e\})|})\\
    =&z^{2c(G)+|E(G)|}(\sum_{A'\in E_2}w^{|A'|+1}z^{-|F_G(A')|-|F_G((A')^c)|}+\sum_{A''\in E_1}w^{|A''|-1}z^{-|F_G(A'')|-|F_G((A'')^c)|})\\
    =&w^{-1}u(w, z)+wv(w, z).
\end{align*}
Here $A'=A\setminus\{e\}$ and $A''=A\cup\{e\}$.

\begin{example}
Consider the ribbon graph $G$ indexed by $11$ in Figure \ref{figure1}, direct calculation shows that
\begin{center}
$^\partial\tilde\varepsilon_{G^{\{e_1\}}}(w, z)=z+2z^2w+zw^2$.
\end{center}
Comparing with the result of $^\partial\tilde\varepsilon_G(w, z)$ in Example \ref{example2.5}, it follows that the partial dual operation does not preserve the graded partial-dual genus polynomial in general.
\end{example}

Although the graded partial-dual genus polynomial is not preserved under partial dual in general, but the next proposition tells us that it is invariant under the classical Euler-Poincar\'{e} dual.

\begin{proposition}\label{proposition7.5}
For any ribbon graph $G$, we have $^\partial\tilde\varepsilon_G(w, z)=^\partial\tilde\varepsilon_{G^*}(w, z)$.
\end{proposition}
\begin{proof}
Assume $|E(G)|=n$ and $^\partial\tilde\varepsilon_G(w, z)=\sum\limits_{i=0}^nf_i(z)w^i$. Then one computes
  \begin{align*}
    ^\partial\tilde\varepsilon_{G^*}(w, z)&=\sum\limits_{A\subseteq E(G^*)}w^{|A|}z^{\varepsilon((G^*)^A)}\\
    &=\sum\limits_{A\subseteq E(G)}w^{|A|}z^{\varepsilon(G^{A^c})}\\
    &=\sum\limits_{A\subseteq E(G)}w^{n-|A|}z^{\varepsilon(G^A)}\\
    &=\sum\limits_{i=0}^nf_{n-i}(z)w^i\\
    &=\sum\limits_{i=0}^nf_i(z)w^i\\
    &=^\partial\tilde\varepsilon_G(w, z)
  \end{align*}
The penultimate equality follows from the fact $f_i(z)=f_{n-i}(z)$, which was given in the end of Section \ref{section2}.
\end{proof}

Later in Subsection \ref{subsection7.4}, we show that $G$ and $G^*$ actually have isomorphic cubes, as well as cochain complices and cohomology groups, while $G$ and $G^A$ may have different cohomology. The former tells us that Proposition \ref{proposition7.5} can be lifted while the latter illustrates that Proposition \ref{proposition7.4} cannot, which shows that the cohomology is strictly stronger than the partial dual genus polynomial itself.

\subsection{The categorification of Proposition \ref{proposition7.1}}
Let $G_1$ and $G_2$ be two disjoint ribbon graphs. Without loss of generality, we can order the face-disks connected in the ribbon-join operation of $G_1\cup G_2$ the first two, and the face-disk obtained in $G_1\vee G_2$ the first one. Then
\begin{center}
$f_S=\mathrm{id}, f_F=\mathrm{id}$, and $f_V=m\otimes \mathrm{id}, f_W=m\otimes \mathrm{id}$
\end{center}
are $S-,F-,V-,W-$cubes maps from $G_1\cup G_2$ to $G_1\vee G_2$. And $f=f_S\otimes f_F\otimes f_V\otimes f_W$ is a cube map from $Cube(G_1\cup G_2)$ to $Cube(G_1\vee G_2)$. Moreover, it induces a surjective chain map from $C(G_1\cup G_2)$ to $C(G_1\vee G_2)$, as well as a homomorphism from $H(G_1\cup G_2)$ to $H(G_1\vee G_2)$, denoted by $f$ and $f^*$ respectively. We have the following short exact sequence
$$0\rightarrow \mathrm{ker}\,f\stackrel{i}{\hookrightarrow}C(G_1\cup G_2)\stackrel{f}{\rightarrow}C(G_1\vee G_2)\{-2\}\rightarrow 0,$$
where the degree shift makes this exact sequence graded. And because $f$ is a cochain map, $\mathrm{ker}\, f$ is a subcochain of $C(G_1\cup G_2)$. Additionally, this short exact sequence gives out a long exact sequence
\begin{equation}
  \cdots\rightarrow H^*(\mathrm{ker}\,f)\stackrel{i^*}{\hookrightarrow}H^*(G_1\cup G_2)\stackrel{f^*}{\rightarrow}H^*(G_1\vee G_2)\{-2\}\stackrel{\delta^*}{\rightarrow}H^{*+1}(\mathrm{ker}\,f)\hookrightarrow\cdots,
  \label{long exact sequence of G1VG2}
\end{equation}
where the map $\delta^*$ is the classical partial operator given by choosing a representative element, finding a preimage of $f$, acted by $\partial_{G_1\cup G_2}$, finding a preimage of $i$, and selecting the corresponding cohomology class.

On the other hand, as
$$(\mathrm{ker}\,f)^i=\mathop\oplus_{\substack{A\in E(G_1\cup G_2)\\|A|=i}}N^{\otimes i}\otimes\mathrm{ker}\,(M^{\otimes 4}\xrightarrow{m\otimes m}M^{\otimes 2}\{-2\})\otimes M^{\otimes |F(A)|+|F(A^c)|-4}\{-2i\},$$
we obtain
$$\chi_q(\mathrm{ker}\,f)=\sum_{A\subseteq E(G_1\cup G_2)}[-q^2(1+p)]^{|A|}(q^{-1}+1+q)^{|F(A)|+|F(A^c)|}[1-(q^{-1}+1+q)^{-2}q^2].$$
Note that the long exact sequence \ref{long exact sequence of G1VG2} shows that
$$\chi_q(G_1\vee G_2)\cdot q^2=\chi_q(G_1\cup G_2)-\chi_q(\mathrm{ker}\,f),$$
what we obtain finally is
$$\chi_q(G_1\vee G_2)=(q^{-1}+1+q)^{-2}\chi_q(G_1\cup G_2).$$
This equality can be used to reprove the first equality of Proposition \ref{proposition7.1}, since
\begin{align*}
  ^\partial\tilde\varepsilon_{G_1\vee G_2}(w, z)
  &=z^{2c(G_1\vee G_2)+|E(G_1\vee G_2)|}\chi_q(G_1\vee G_2)\big|_{-q^2-pq^2=w, q^{-1}+1+q=z^{-1}}\\
  &=z^{2c(G_1\cup G_2)-2+|E(G_1\cup G_2)|}(q^{-1}+1+q)^{-2}\chi_q(G_1\cup G_2)\big|_{-q^2-pq^2=w, q^{-1}+1+q=z^{-1}}\\
  &=z^{2c(G_1\cup G_2)+|E(G_1\cup G_2)|}\chi_q(G_1\cup G_2)\big|_{-q^2-pq^2=w, q^{-1}+1+q=z^{-1}}\\
  &=^\partial\tilde\varepsilon_{G_1\cup G_2}(w, z).
\end{align*}
This is to say, the long exact sequence (\ref{long exact sequence of G1VG2}) categorifies the result of Proposition \ref{proposition7.1}.

\subsection{The categorification of Proposition \ref{proposition7.2}}
Suppose ribbon graphs $G_1$ and $G_2$ are disjoint again and we order the face-disks of $G_1\cup G_2$ and $G_1\eqcirc G_2$ similar to the former subsection. In other words, we order the face-disks connected in the bar-amalgamation operation of $G_1\cup G_2$ the first two and the face-disk obtained in $G_1\eqcirc G_2$ the first one. Then we can take the cochain group of $G_1\eqcirc G_2$ into two parts as
\begin{align*}
  C^i(G_1\eqcirc G_2)=&\mathop\oplus_{\substack{A\subseteq E(G_1\eqcirc G_2)\\|A|=i}}N^i\otimes M^{F_\eqcirc(A)}\otimes M^{F_\eqcirc(A^c)}\{-2i\}\\
  =&(\mathop\oplus_{\substack{e\in A\subseteq E(G_1\eqcirc G_2)\\|A|=i}}N^i\otimes M^{F_\eqcirc(A)}\otimes M^{F_\eqcirc(A^c)}\{-2i\})\\
  \oplus&(\mathop\oplus_{\substack{e\notin A\subseteq E(G_1\eqcirc G_2)\\|A|=i}}N^i\otimes M^{F_\eqcirc(A)}\otimes M^{F_\eqcirc(A^c)}\{-2i\}),
\end{align*}
where $e$ is the edge added in the bar-amalgamation operation and $F_\eqcirc$ counts the number of the face-disks of the corresponding subribbon graphs of $G_1\eqcirc G_2$. Note that when $e\in A$, $F_\eqcirc(A)=F_\vee(A-e)=F_\cup (A-e)-1$, $F_\eqcirc(A^c)=F_\vee(A^c)+1=F_\cup (A^c)$, and when $e\notin A$, $F_\eqcirc(A)=F_\vee(A)+1=F_\cup (A)$, $F_\eqcirc(A^c)=F_\vee(A^c)=F_\cup (A^c)-1$, we can construct following cube maps:
\begin{align*}
  g_S&=\mathrm{id},\quad   g_F=\mathrm{id}, \quad g_V=\mathrm{id}, \quad  g_W=m\otimes \mathrm{id};\\
  h_S&=\mathrm{id},\quad   h_F=\mathrm{id}, \quad  h_V=m\otimes \mathrm{id}, \quad  h_W=\mathrm{id};\\
  r_S&=\mathrm{id},\quad   r_F=\mathrm{id}, \quad  r_V=m\otimes \mathrm{id}, \quad r_W=\mathrm{id};\\
  t_S&=-\mathrm{id},\quad   t_F=\mathrm{id}, \quad t_V=\mathrm{id}, \quad  t_W=m\otimes \mathrm{id},
\end{align*}
where $g$ is the cube map from the subcube of $G_1\eqcirc G_2$ with edge set containing $e$ to the cube of $G_1\vee G_2$, $h$  is the cube map from the subcube of $G_1\eqcirc G_2$ with edge set rejecting $e$ to the cube of $G_1\vee G_2$, $r$ is the cube map from the cube of $G_1\cup G_2$ to the subcube of $G_1\eqcirc G_2$ with edge set containing $e$, and $t$ is the cube map from the cube of $G_1\cup G_2$ to the subcube of $G_1\eqcirc G_2$ with edge set rejecting $e$.

Based on the properties of the TQFT, it can be easily verified that $g+h$ and $r+t$ induce two graded cochain map $g+h:C(G_1\eqcirc G_2)\to C(G_1\vee G_2)\{-1\}$ and $r+t:C(G_1\cup G_2)\to C(G_1\eqcirc G_2)\{-1\}$. Furthermore, if we extend one of these two maps to a short exact sequence, we will obtain a long exact sequence of cohomology groups which categorifies the Proposition \ref{proposition7.2}. The progress is quite the same as the one in the former subsection.

\begin{remark}
Note that $(g+h)\circ(r+t)=g\circ r+h\circ t=0$, $g+h$ is surjective and $r+t$ is injective, we obtain a graded short sequence \textbf{not} exact only at the middle term as
\begin{equation*}
  0\to C(G_1\cup G_2)\stackrel{r+t}{\longrightarrow}C(G_1\eqcirc G_2)\{-1\}\stackrel{g+h}{\longrightarrow}C(G_1\vee G_2)\{-2\}\to 0.
  \label{long exact sequence of G1cupG2}
\end{equation*}
In a sense, this sequence is not so satisfactory. We guess there is no short exact sequence utilize these three cochain groups.
\end{remark}

\subsection{The categorification of Proposition \ref{proposition7.5}}\label{subsection7.4}
For a given ribbon graph $G$ and a subset $A\subseteq E(G)$, the corresponding vertex space is $Cube(G)_A=N^{\otimes |A|}\otimes M^{F_G(A)}\otimes M^{F_G(A^c)}$, while for its dual graph $G^*$, it is $Cube(G^*)_A=N^{\otimes |A|}\otimes M^{\otimes F_{G^*}(A)}\otimes M^{\otimes F_{G^*}(A^c)}$. Notice that
\begin{center}
$F_G(A)=V(G^A)=V((G^*)^{A^c})=F_{G^*}(A^c), F_G(A^c)=V(G^{A^c})=V((G^*)^A)=F_{G^*}(A)$.
\end{center}
Now we define a linear map between these two vertex spaces by
\begin{align*}
  P_A: &Cube(G)_A\to Cube(G^*)_A,\\
  &n_1\otimes\cdots\otimes n_{|A|}\otimes m_1\otimes\cdots\otimes m_{|F_G(A)|}\otimes m^c_1\otimes \cdots \otimes m^c_{|F_G(A^c)|}\\
  \mapsto &n_1\otimes\cdots\otimes n_{|A|}\otimes m^c_1\otimes \cdots \otimes m^c_{|F_G(A^c)|}\otimes m_1\otimes\cdots\otimes m_{|F_G(A)|}.
\end{align*}
It is routine to verify that $P_A$ is an isomorphism and commutative with the edge maps of these two cubes, which means that it induces a cube isomorphism. This clarifies the following proposition.
\begin{proposition}
For any ribbon graph $G$,
\begin{center}
$Cube(G)\cong Cube(G^*), C(G)\cong C(G^*)$ and $H(G)\cong H(G^*)$.
\end{center}
\end{proposition}
The third isomorphism can be seen as a categorification of the equality in Proposition \ref{proposition7.5}. However, the example below shows that this result does not hold for general partial dual graphs.

\begin{figure}[htbp]
  \centering
  \begin{tikzpicture}[scale=0.7, use Hobby shortcut, baseline=-3pt]
    \draw (10:1) arc (10:170:1);
    \draw (-10:1) arc (-10:-170:1);
    \draw (10:1)--++(1,0) (-10:1)--++(1,0);
    \node at (1.5,0.5) {$e_2$};
    \draw (170:1)--++(-1,0) (-170:1)--++(-1,0);
    \node at (-1.5,0.5) {$e_1$};
    \draw ($(10:1)+(-3,0)$) arc (10:350:1);
    \draw ($(-10:1)+(1,0)$) arc (-170:170:1);
    \node at (0,-1.5) {$G$};
  \end{tikzpicture}
  \quad
  \begin{tikzpicture}[scale=0.7, use Hobby shortcut, baseline=-3pt]
    \draw (55:1) arc (55:180-55:1);
    \draw (180-35:1) arc (180-35:350:1);
    \draw (10:1) arc (10:35:1);
    \draw ($(-10:1)+(1,0)$) arc (-170:170:1);
    \draw (55:1)..(90:2)..(180-55:1);
    \draw (35:1)..(90:2.3)..(180-35:1);
    \node at (-1.5,1.5) {$e_1$};
    \draw (10:1)--++(1,0) (-10:1)--++(1,0);
    \node at (1.5,0.5) {$e_2$};
    \node at (1,-1.5) {$G^{\{e_1\}}$};
  \end{tikzpicture}
  \caption{A ribbon graph and one of its partial dual graphs}
  \label{figure11}
\end{figure}
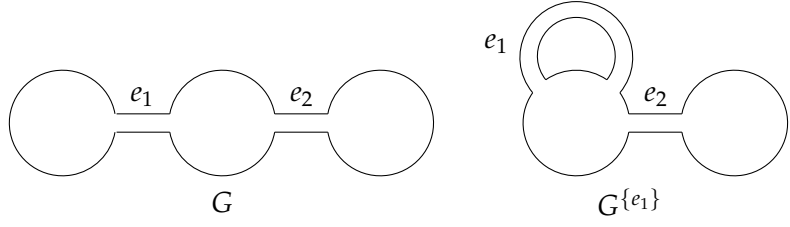

\begin{example}\label{example2}
Let $G$ be the ribbon graph on the left side of Figure \ref{figure11}, with its two ribbons denoted by $e_1$ and $e_2$, then one of its partial dual graph $G^{\{e_1\}}$ is as the right one. Their cubes are illustrated in Figure \ref{figure12} and Figure \ref{figure13}.

\begin{figure}[h]
    \centering
  \begin{tikzcd}
    &  & N\otimes M^{\otimes 2}\otimes M^{\otimes 2}\arrow[rrd, "u\otimes \mathrm{Id}_N\otimes m\otimes\Delta\otimes\mathrm{Id}_M"] &  &   \\
  M^{\otimes 3}\otimes M\arrow[rrd, "u\otimes m\otimes\mathrm{Id}_M\otimes\Delta"] \arrow[rru, "u\otimes\mathrm{Id}_M\otimes m\otimes\Delta"] &  & &  & N^{\otimes 2}\otimes M\otimes M^{\otimes 3} \\
    &  & N\otimes M^{\otimes 2}\otimes M^{\otimes 2} \arrow[rru, "-u\otimes\mathrm{Id}_N\otimes m\otimes\mathrm{Id}_M\otimes \Delta"] &  &
  \end{tikzcd}
    \caption{$Cube(G)$}
    \label{figure12}
\end{figure}

\begin{figure}[h]
    \centering
  \begin{tikzcd}
    &  & N\otimes M\otimes M^{\otimes 3}\arrow[rrd, "u\otimes \mathrm{Id}_N\otimes \Delta\otimes m\otimes\mathrm{Id}_M"] &  &   \\
  M^{\otimes 2}\otimes M^{\otimes 2}\arrow[rrd, "u\otimes\Delta\otimes \mathrm{Id}_M\otimes m"] \arrow[rru, "u\otimes m\otimes\mathrm{Id}_M\otimes\Delta"] &  & &  & N^{\otimes 2}\otimes M^{\otimes 2}\otimes M^{\otimes 2} \\
    &  & N\otimes M^{\otimes 3}\otimes M \arrow[rru, "-u\otimes\mathrm{Id}_N\otimes\mathrm{Id}_M\otimes m\otimes \Delta"] &  &
  \end{tikzcd}
    \caption{$Cube(G^{e_1})$}
    \label{figure13}
\end{figure}
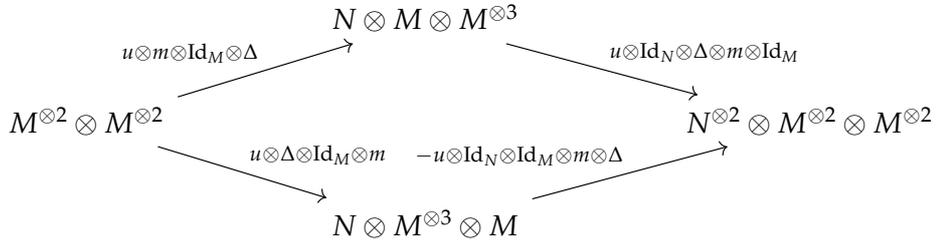

We can calculate their cohomology groups with index $0$ by
    \begin{align*}
      H^0(G)=&\ker(\mathrm{Id}_M\otimes m\otimes  \Delta)\cap \ker(m\otimes \mathrm{Id}_M\otimes \Delta)\\
      =&(M\otimes\ker m\otimes M)\cap (\ker m\otimes M\otimes M)\\
      =&[(M\otimes\ker m)\cap (\ker m\otimes M)]\otimes M,
    \end{align*}
    \begin{align*}
      H^0(G^{\{e_1\}})=&\ker(m\otimes \mathrm{Id}_M\otimes \Delta)\cap \ker(\Delta\otimes\mathrm{Id}_M\otimes m)\\
      =&(\ker m\otimes M\otimes M)\cap (M\otimes M\otimes\ker m)\\
      =&\ker m\otimes \ker m.
    \end{align*}
We only concentrate on the subgroup with bi-degree $(0,2)$, which shows these two homology groups are different directly. In fact, recall that the degree of $1,x$ and $x^2$ are $1,0$ and $-1$ respectively, we have
    \begin{align*}
      (\ker m)_{-2}=&\langle x^2\otimes x^2\rangle,\\
      (\ker m)_{-1}=&\langle x^2\otimes x, x\otimes x^2\rangle,\\
      (\ker m)_{0}=&\langle 1\otimes x^2-x^2\otimes 1,1\otimes x^2-x\otimes x\rangle,\\
      (\ker m)_1=&\langle 1\otimes x-x\otimes 1\rangle,\\
      (\ker m)_{\neq -2,-1,0,1}=&0,
    \end{align*}
where the subscripts means the second degree and the first degree is omitted as it always being $0$. And based on the graded structure of $\ker m$, it can be calculated that
    $$(\ker m\otimes M)_1=\langle(1\otimes x-x\otimes 1)\otimes x,(1\otimes x^2-x^2\otimes 1)\otimes 1,(1\otimes x^2-x\otimes x)\otimes 1\rangle,$$
    $$(M\otimes \ker m)_1=\langle x\otimes(1\otimes x-x\otimes 1),1\otimes (1\otimes x^2-x^2\otimes 1),1\otimes(1\otimes x^2-x\otimes x)\rangle.$$
Note that these six elements spanning $\ker m\otimes M$ and $M\otimes \ker m$ are linear independent in $M^{\otimes 3}$, hence $(\ker m\otimes M)_1\cap (M\otimes \ker m)_1=0$, which further implies that $H^0_{(0,2)}(G)=(\ker m\otimes M)_1\cap (M\otimes \ker m)_1\otimes M_1=0.$ While $(\ker m)_1=\mathbb{Z}[\sqrt{3}]$ shows that $H^0_{(0,2)}(G^{\{e_1\}})=\mathbb{Z}[\sqrt{3}]\otimes \mathbb{Z}[\sqrt{3}]=\mathbb{Z}[\sqrt{3}]$. This means $H^0(G)$ and $H^0(G^{\{e_1\}})$ are different as graded modules, so that these two ribbon graphs have different cohomologies.

On the other hand, observe Figure~\ref{figure12} and Figure~\ref{figure13} and it can be found that the vertex space on each vertex of their cubes are isomorphic, hence the ribbon graph $G$ and $G^{\{e_1\}}$ own the same graded partial-dual genus polynomial. So this example shows that the cohomology is strictly stronger than the graded partial-dual genus polynomial.
\end{example}

\section{Some remarks on algebraic structures}\label{section8}
There are plenty of algebraic structures can be used to construct a categorification of the partial-dual genus polynomial. Fix the Frobenius algebra as before, we may attempt different half genus maps. While it turns out that essentially the half genus map defined in Section \ref{section3} is the unique one able to realize this categorification. And actually, we can construct a Frobenius algebra with half genus map for any odd dimension greater than three.

\subsection{Other half genus maps}
Recall that the half genus map $h$ is an endomorphism of Frobenius algebra $M$ with condition $h^2=m\circ \Delta$. In order to construct a TQFT such that it corresponds to a punctured annulus, it needs to satisfy
$$m\circ(h\otimes\mathrm{Id}_M)=h\circ m=m\circ(\mathrm{Id}_M\otimes h),$$
$$(h\otimes\mathrm{Id}_M)\circ\Delta=\Delta\circ h=(\mathrm{Id}_M\otimes h)\circ\Delta.$$
The first equation shows that $h\in\mathrm{End}_M(M)$, which is to say it is an multiplication with a fix element belongs to $M$. Denote this element by $h=h_0+h_1x^1+h_2x^2$, then it is easy to check that these two equations are satisfied for any $h_1,h_2,h_3\in\mathbb{Z}[\sqrt{3}]$. While to fit the half genus relation $h^2=m\circ \Delta$, after some elementary calculations, we obtain that $h_0=0,h_1=\pm\sqrt{3}$ and $h_2$ can be any element in $\mathbb{Z}[\sqrt{3}]$.

Then for general $h_2$, we can construct the punctured TQFT as before and the cochain complex whose graded Euler characteristic is the partial-dual genus polynomial, with a little modification. While for nonzero $h_2$, the half genus map is no longer a graded map, not to mention that it owns the same degree with $m$ and $\Delta$, hence it cannot be derived that the graded Euler characteristic of the cohomology groups equals to that of the cochain groups.

However, just as the case that another differential in Khovanov homology theory brings out Lee's endomorphism\cite{Lee2005}, if we still use $\partial$ to denote the differential of the cochain complex given by $h=\sqrt{3}x$ while $\partial'$ by $h'=\sqrt{3}x+x^2$, then a new differential $\Phi:=\partial'-\partial$ is given by zero multiplication, unit, comultiplication, trace while only half genus map being a multiplication with $x^2$. Using this differential, we can construct a bicomplex $(C(G),\partial,\varPhi)$ whose spectral sequence is given by horizontal (or vertical, resp.) filtration, say $E_{hor}$(or $E_{vert}$, resp.). And both of them converge to $H(C(G),\partial')$, the cohomology induced by $h'$.

\begin{remark}
If a ribbon graph $G$ is orientable, then both $E_{hor}$ and $E_{vert}$ converge at the second page, since $\Phi=0$ in this case. We define the \emph{orientability number} of $G$ to be the minimal number of edges such that removing these edges yields an orientable surface. We guess that the orientability number of a ribbon graph $G$ might be read from some information of $E_{hor}$ and $E_{vert}$. For example, the difference of their second pages, as well as the speed of convergence. While maybe a computer program is in need to provide valid data support.
\end{remark}

\subsection{Other Frobenius algebras}
We can also use other algebraic structures to categorify the partial-dual genus polynomial. This is different from Khovanov's categorification of the Jones polynomial \cite{Kho2000} and Helme-Guizon and Rong's categorification of the chromatic polynomial \cite{HR2005}, because the former one lifts the skein relation of the Jones polynomial and the latter one lifts the deletion-contraction rule of the chromatic polynomial.

In fact, we can choose different $M$, such as $\mathbb{Z}[\sqrt{2n+1},x]/x^{2n+1}$, with Frobenius trace $\epsilon$ mapping each polynomial to the coefficient before $x^{2n}$ and the half genus map being the multiplication with $\sqrt{2n+1}x^n$, and $\Phi$ to be the map induced by multiplication with $x^{n+i}$ for any $i>0$ to construct the bicomplex as well as the spectral sequence. While this will make the calculation much more complicated.

\begin{remark}
  Unlike Khovanov homology \cite{Kho2000} needs the existence of the unit map to preserve the homology groups under three kinds of Reidemeister moves, as we do not need the unit map of the Frobenius algebra, we can select a subalgebra without unit, and maybe the simplest one is the ideal generated by the half genus element, the one multiplied by the half genus map. The homology groups given in this way is much easier to calculate, but it may lost some properties as the TQFT does not work anymore.
\end{remark}

\section*{Acknowledgement}
Ziyi Lei realized that the half genus map is important and attempted to construct it days before the workshop on cluster algebra, representation theory and algebraic geometry at HongKong University during the summer in 2023. And he discussed this problem with Jingfang Deng during the workshop. And Ziyi Lei would especially appreciate for Deng's valuable suggestions, as well as the organizers' reception for this workshop. Ziyi Lei would also thank Pengyun Chen, Jingcheng Li, Chenhui Liu, Bohan Xing, Guanzhen Zhang for discussion. The authors are grateful to Yifan Jin for indicating the paper \cite{TT2006} to us. Zhiyun Cheng was supported by the NSFC grant 12371065 and NSFC grant 12071034.


\begin{thebibliography}{0}
\bib{AN2006}{article}{
	author={Alexeevski, A.},
	author={Natanzon, S.},
	title={Noncommutative two-dimensional topological field theories and {H}urwitz numbers for real algebraic curves},
	journal={Selecta Math. (N.S.)},
	volume={12},
	year={2006},
	number={3-4},
	pages={307--377}}

\bib{Bar2002}{article}{
	author={D. Bar-Natan},
	title={On Khovanov's categorification of the Jones polynomial},
	journal={Algebraic \& Geometry Topology},
	volume={2},
	year={2002},
	number={1},
	pages={337--370}}

\bib{CLWZ2022}{article}{
	author={Cheng, Zhiyun},
	author={Lei, Ziyi},
	author={Wang, Yitian},
	author={Zhang, Yanguo},
	title={A categorification for the signed chromatic polynomial},
	journal={Electron. J. Combin.},
	volume={29},
	year={2022},
	number={2}
	pages={Paper No. 2.49, 26}}

\bib{Chm2009}{article}{
	author={S. Chmutov},
	title={Generalized duality for graphs on surfaces and the signed Bollob\'as-Riordan polynomial},
	journal={Journal of Combinatorial Theory, Series B},
	volume={99},
	year={2009},
	pages={717--638}}
	
\bib{Chm2023}{article}{
	author={S. Chmutov},
	title={Partial-dual genus polynomial as a weight system},
	journal={arXiv:2307.15211v2}}

\bib{GMT2020}{article}{
	author={J. Gross},
	author={T. Mansour},
	author={T. Tucker}
	title={Partial duality for ribbon graphs, {I}: Distributions},
	journal={European Journal of Combinatorics},
	volume={86},
	year={2020},
	pages={103084}}
	
\bib{GMT2021}{article}{
	author={J. Gross},
	author={T. Mansour},
	author={T. Tucker}
	title={Partial duality for ribbon graphs, {II}: partial-twuality polynomials and monodromy computations},
	journal={European Journal of Combinatorics},
	volume={95},
	year={2021},
	pages={103329}}
	
\bib{GMT2021+}{article}{
	author={J. Gross},
	author={T. Mansour},
	author={T. Tucker}
	title={Partial duality for ribbon graphs, {III}: a {G}ray code algorithm for enumeration},
	journal={J. Algebraic Combin.},
	volume={54},
	year={2021},
	number={4},
	pages={1119--1135}}

\bib{HR2005}{article}{
  title={A categorification for the chromatic polynomial},
  author={Helme-Guizon, Laure},
  author={Rong, Yongwu},
  journal={Algebraic \& Geometric Topology},
  year={2005},
  volume={5},
  pages={1365--1388}}

\bib{JR2006}{article}{
  title={A categorification for the Tutte polynomial},
  author={Jasso-Hernandez, Edna F.},
  author={Rong, Yongwu},
  journal={Algebraic \& Geometric Topology},
  year={2006},
  volume={6},
  pages={2031--2049}}

\bib{Kho2000}{article}{
	author={M. Khovanov},
	title={A catgorification of the Jones polynomial},
	journal={Duke Mathematical Journal},
	volume={101},
	year={2000},
	number={3},
	pages={360--413}}
	
\bib{Kho2004}{article}{
	author={M. Khovanov},
	title={sl(3) link homology},
	journal={Algebraic \& Geometry Topology},
	volume={4},
	year={2004},
	number={2},
	pages={1045--1081}}

\bib{Kho2005}{article}{
    author={M. Khovanov},
    title={Categorifications of the colored Jones polynomial},
    journal={J. Knot Theory Ramifications},
    volume={14},
    year={2005},
    number={1},
    pages={111--130}}

\bib{KR2008}{article}{
  author={M. Khovanov},
  author={L. Rozansky},
  title={Matrix factorizations and link homology},
  journal={Fund. Math.},
  volume={199},
  year={2008},
  number={1},
  pages={1--91}}

\bib{KR2008II}{article}{
   author={M. Khovanov},
   author={L. Rozansky},
   title={Matrix factorizations and link homology. II},
   journal={Geom. Topol.},
   volume={12},
   year={2008},
   number={3},
   pages={1387--1425}}

\bib{Kock2004}{book}{
author={J. Kock},
title={Frobenius algebras and 2D topological quantum field theories},
series={London Mathematical Society Student Texts},
volume={59},
publisher={Cambridge University Press, Cambridge},
year={2004},
pages={xiv+240}}

\bib{Lee2005}{article}{
author = {E. S. Lee}
  title = {An endomorphism of the Khovanov invariant},
  journal = {Advances in Mathematics},
  volume = {197},
  number = {2},
  pages = {554-586},
  year = {2005}}

\bib{Mof2013}{article}{
  author = {I. Moffatt}
  title = {Separability and the genus of a partial dual},
  journal = {European Journal of Combinatorics},
  volume = {34},
  number = {2},
  pages = {355-378},
  year = {2013}}

\bib{Pic2020}{article}{
  author = {Piccirillo, Lisa}
  title = {The {C}onway knot is not slice},
  journal = {Ann. of Math. (2)},
  volume = {191},
  year = {2020},
  number = {2},
  pages = {581--591}}

\bib{Ras2010}{article}{
  author = {Rasmussen, Jacob}
  title = {Khovanov homology and the slice genus},
  journal = {Invent. Math.},
  volume = {182},
  year = {2010},
  number = {2},
  pages = {419--447}}

\bib{SY2018}{article}{
   author={Sazdanovic, Radmila},
   author={Yip, Martha}
   title={A categorification of the chromatic symmetric function},
   journal={J. Combin. Theory Ser. A},
   volume={154},
   year={2018},
   pages={218--246}}

\bib{TT2006}{article}{
   author={Turaev, Vladimir},
   author={Turner, Paul}
   title={Unoriented topological quantum field theory and link homology},
   journal={Algebraic \& Geometry Topology},
   volume={6},
   year={2006},
   pages={1069--1093}}

\bib{VV2007}{incollection}{
author={V. Vershinin},
author={A. Vesnin},
title={Yamada polynomial and Khovanov cohomology},
booktitle={Intelligence of low dimensional topology 2006},
series={Ser. Knots Everything},
volume={40},
pages={337--346},
publisher={World Sci. Publ., Hackensack, NJ},
year={2007}}

\bib{YJ2022}{article}{
   author={Yan, Qi},
   author={Jin, Xian'an}
   title={Partial-dual polynomials and signed intersection graphs},
   journal={Forum Math. Sigma},
   volume={10},
   year={2022},
   pages={Paper No. e69, 16}}
\end{thebibliography}
\end{document}